\documentclass[12pt]{amsart}

\usepackage{fullpage}
\usepackage{amsmath}
\usepackage{amsthm}
\usepackage{amssymb}
\usepackage{amsrefs}
\usepackage{mathrsfs}
\usepackage{color}
\usepackage{tikz}
\usepackage{hyperref}
\usepackage{stmaryrd}
\usetikzlibrary{arrows}
\usetikzlibrary{positioning}
\usetikzlibrary{cd}

\newcommand{\mfrak}{\mathfrak{m}}

\newcommand{\sfrak}{\mathfrak{s}}

\newcommand{\ufrak}{\mathfrak{u}}

\newcommand{\Lfrak}{\mathfrak{L}}

\newcommand{\Abf}{\mathbf{A}}

\newcommand{\Gbf}{\mathbf{G}}

\newcommand{\Lbf}{\mathbf{L}}
\newcommand{\Mbf}{\mathbf{M}}

\newcommand{\Dcal}{\mathcal{D}}

\newcommand{\Gcal}{\mathcal{G}}
\newcommand{\Hcal}{\mathcal{H}}
\newcommand{\Ical}{\mathcal{I}}

\newcommand{\Lcal}{\mathcal{L}}

\newcommand{\Ocal}{\mathcal{O}}

\newcommand{\Rcal}{\mathcal{R}}

\newcommand{\Dscr}{\mathscr{D}}

\newcommand{\Kscr}{\mathscr{K}}
\newcommand{\Lscr}{\mathscr{L}}
\newcommand{\Mscr}{\mathscr{M}}

\newcommand{\Uscr}{\mathscr{U}}
\newcommand{\Vscr}{\mathscr{V}}

\newcommand{\Frm}{\mathrm{F}}

\newcommand{\Trm}{\mathrm{T}}
\newcommand{\Urm}{\mathrm{U}}

\newcommand{\Abb}{\mathbb{A}}
\newcommand{\Cbb}{\mathbb{C}}
\newcommand{\Dbb}{\mathbb{D}}

\newcommand{\Gbb}{\mathbb{G}}

\newcommand{\Pbb}{\mathbb{P}}
\newcommand{\Qbb}{\mathbb{Q}}

\newcommand{\Zbb}{\mathbb{Z}}

\newcommand{\Ext}{\operatorname{Ext}}

\newcommand{\Spec}{\operatorname{Spec}}

\newcommand{\smin}{\smallsetminus}

\newcommand{\HH}{\operatorname{H}}

\newcommand{\Aut}{\operatorname{Aut}}

\newcommand{\Div}{\operatorname{Div}}
\newcommand{\one}{\mathbf{1}}

\newcommand{\Hom}{\operatorname{Hom}}

\newcommand{\image}{\operatorname{image}}

\newcommand{\sing}{\mathrm{sing}}
\newcommand{\an}{\mathrm{an}}
\newcommand{\trdeg}{\operatorname{trdeg}}
\newcommand{\KM}{\operatorname{K}^\mathrm{M}}
\newcommand{\Lie}{\operatorname{Lie}}
\newcommand{\MHS}{\mathrm{MHS}}
\newcommand{\Pic}{\operatorname{Pic}}
\newcommand{\Picbf}{\mathbf{Pic}}
\newcommand{\acl}{\operatorname{acl}}
\newcommand{\KL}{\Kscr_{\Lambda}}
\renewcommand{\div}{\operatorname{div}}
\newcommand{\can}{\mathrm{can}}
\newcommand{\coker}{\operatorname{coker}}

\newtheorem{theorem}{Theorem}
\newtheorem{lemma}[theorem]{Lemma}
\newtheorem{proposition}[theorem]{Proposition}
\newtheorem{fact}[theorem]{Fact}
\newtheorem{step}{Step}

\newtheorem*{theorem*}{Theorem}
\newcommand{\Isom}{\operatorname{Isom}}
\newcommand{\rat}{\mathrm{rat}}
\newcommand{\Isommod}{\underline{\Isom}}

\newtheorem{maintheorem}{Theorem}

\theoremstyle{remark}

\renewcommand{\tilde}{\widetilde}
\renewcommand{\epsilon}{\varepsilon}
\renewcommand{\bar}{\overline}

\numberwithin{theorem}{section}
\numberwithin{equation}{section}

\title{A Torelli theorem for higher-dimensional function fields}
\author{Adam Topaz}
\address{
    Mathematical and Statistical Sciences \\
    University of Alberta \\
    632 Central Academic Building \\
    Edmonton AB T6G 2G1 \\
    Canada }
\email{topaz@ualberta.ca}
\urladdr{https://adamtopaz.com/}
\date{\today}
\thanks{Research supported in part by EPSRC programme grant EP/M024830/1, and in part by NSERC and the University of Alberta.}

\begin{document}

\begin{abstract}
  We prove a Torelli-like theorem for higher-dimensional function fields, from the point of view of ``almost-abelian'' anabelian geometry.
\end{abstract}
\maketitle

\setcounter{tocdepth}{1}
\tableofcontents

\section{Introduction}\label{section:intro}

The classical \emph{Torelli Theorem}, in its cohomological form, can be stated as follows:

\begin{theorem*}
  Let $X$ be a smooth compact complex curve.
  Then the isomorphism type of $X$ is determined by the singular cohomology group $\HH^1(X,\Zbb)$, endowed with its canonical polarized Hodge structure.
\end{theorem*}

In this paper, we develop and prove a higher-dimensional \emph{birational} variant of this theorem.
As expected, one must include not only $\HH^1$ (with its \emph{mixed} Hodge structure), but some additional \emph{non-abelian data} as well.
It turns out that the ``two-step nilpotent'' information, encoded as the kernel of the cup-product, provides sufficient non-abelian data in this setting.
Also, our result works even with \emph{rational coefficients}, in contrast to the classical Torelli theorem mentioned above.

\subsection{Main result}\label{subsection:main-result}

Let $k$ be an algebraically closed field and $\sigma : k \hookrightarrow \Cbb$ an embedding into the complex numbers.
Let $\Lambda$ be any subring of $\Qbb$.
For a $k$-variety $X$, consider $X^{\an} := X(\Cbb)$ (computed via $\sigma$) endowed with the complex topology, and define the \emph{Betti} cohomology of $X$ in the usual way as
\[ \HH^{i}(X,\Lambda) := \HH^{i}_{\sing}(X^{\an},\Lambda). \]
Following Deligne~\cites{MR0498551,MR0498552}, one can endow $\HH^{i}(X,\Lambda)$ with a canonical \emph{mixed Hodge structure} over $\Lambda$.
We write $\Lambda(j)$ for the $j$-th Tate-Hodge structure over $\Lambda$, which is the unique pure Hodge structure over $\Lambda$ with underlying module $\Lambda$ which is of Hodge type $(-j,-j)$.
As usual, we write $\HH^{i}(X,\Lambda(j)) := \HH^{i}(X,\Lambda) \otimes \Lambda(j)$.

Let $K|k$ be a function field, and $X$ a \emph{model} of $K|k$, i.e.~$X$ is an integral $k$-variety whose function field is $K$.
Define
\[ \HH^{i}(K|k,\Lambda(j)) := \varinjlim_{U} \HH^{i}(U,\Lambda(j)), \]
where $U$ varies over the nonempty open $k$-subvarieties of $X$.
We consider $\HH^{i}(K|k,\Lambda(j))$ as a mixed Hodge structure of possibly infinite rank.
This construction does not depend on the choice of model $X$ of $K|k$.

The cup-product in singular cohomology yields a well-defined cup-product on
\[ \HH^{*}(K|k,\Lambda(*)) := \bigoplus_{i \geq 0} \HH^{i}(K|k,\Lambda(i)), \]
making it into a graded-commutative ring.
We write
\[ \Rcal(K|k,\Lambda) := \{ (x,y) \ | \ x,y \in \HH^{1}(K|k,\Lambda(1)), \ x \cup y = 0 \}\]
for the set of pairs of elements of $\HH^{1}(K|k,\Lambda(1))$ whose cup-product vanishes.
With this notation and terminology, we may state our main result as follows (see Theorem~\ref{theorem:main-theorem} for a precise formulation).

\begin{maintheorem}\label{maintheorem}
  In the above context, assume furthermore that $\trdeg(K|k) \geq 2$.
  Then the isomorphism type of $K|k$, as fields, is determined by the following data:
  \begin{enumerate}
    \item The mixed Hodge structure $\HH^{1}(K|k,\Lambda(1))$.
    \item The subset $\Rcal(K|k,\Lambda) \subset \HH^{1}(K|k,\Lambda(1)) \times \HH^{1}(K|k,\Lambda(1))$.
  \end{enumerate}
\end{maintheorem}

\subsection{A comment about the proof}

Theorem~\ref{maintheorem} should certainly seem reasonable, especially to the reader who is familiar both with results concerning $1$-motives and their Hodge realizations, and with certain recent results from birational anabelian geometry.
The main result follows by combining the following:
\begin{enumerate}
  \item The comparison of a $1$-motive with its Hodge realization, due to Deligne~\cite{MR0498552}.
        We will only need a simple case of this comparison which is summarized in Lemma~\ref{lemma:hodge-realization-iso}.
  \item The construction of the Picard $1$-motive of a smooth quasi-projective variety and the calculation of its Hodge realization.
        This is due essentially to Serre~\cite{serre1958morphismes2}, and/or the works of Barbieri-Viale, Srinivas~\cite{MR1891270}, Ramachandran~\cite{MR1837235}.
  \item Methods for reconstructing function fields over algebraically closed fields in birational anabelian geometry, due to Bogomolov~\cite{MR1260938}, Bogomolov-Tschinkel~\cites{MR2421544,MR2537087} and Pop~\cites{pop2002birational,MR2867932,MR2891876}.
\end{enumerate}
In addition to the above points, there are several nontrivial hurdles one must overcome, specifically in the case where $\Lambda = \Qbb$, where the known ``global'' anabelian techniques (e.g.~from Pop~\cites{MR2867932,MR2891876} and/or Bogomolov-Tschinkel~\cites{MR2421544,MR2537087}) break down, as one can no longer distinguish between the ``divisible'' and ``non-divisible.''
We overcome these difficulties by relying on arguments surrounding the connection between algebraic dependence and the cup-product, which are, in some sense, analogous to the ideas from~\cite{MR3552242} and~\cite{MR3827205}.
As we see it, the primary novelty of this work comes from the fact that it applies \emph{anabelian} techniques in a purely \emph{motivic} setting.

\subsection*{Acknowledgments}

The author warmly thanks T.~Szamuely for several comments and suggestions concerning an early version of this paper.
The author also thanks all who expressed interest in this work, and in particular A.~Cadoret, M.~Kim, K.~Kremnitzer, F.~Pop, and B.~Zilber.

\section{Cohomology}\label{section:cohom}

Throughout the paper, we work with a fixed algebraically closed field $k$ endowed with an embedding $\sigma : k \hookrightarrow \Cbb$, and a subring $\Lambda$ of $\Qbb$ which will be fixed as our coefficient ring.
By a \emph{$k$-variety} we mean a separated scheme of finite type over $k$, and \emph{morphisms} of $k$-varieties are morphisms over $\Spec k$.

Given a $k$-variety $X$, write $X^{\an} := X(\Cbb)$, computed via $\sigma$, and endowed with the complex topology.
We will work with \emph{Betti cohomology} (with respect to $\sigma$) defined in the usual way as the singular cohomology of $X^{\an}$:
\[ \HH^{i}(X,\Lambda) := \HH^{i}_{\sing}(X^{\an},\Lambda). \]
This is a $\Lambda$-module which is canonically endowed with a mixed Hodge structure~\cites{MR0498551,MR0498552}.
We write $\HH^{i}(X,\Lambda(j)) := \HH^{i}(X,\Lambda) \otimes \Lambda(j)$ for its $j$-th Tate twist.
Of course, the construction of $\HH^{i}(X,\Lambda(j))$ depends on the choice of $\sigma$, but we will exclude it from the notation while ensuring it is understood from context.

\subsection{Models}\label{subsection:models}

Let $K$ be a function field over $k$.
By a \emph{model} of $K|k$, we mean an integral quasi-projective $k$-variety whose function field is $K$.
Given such a model, we define
\[ \HH^{i}(K|k,\Lambda(j)) := \varinjlim_{U} \HH^{i}(U,\Lambda(j)), \]
where $U$ varies over the nonempty open $k$-subvarieties of $X$, considered as a (possibly infinite-rank) mixed Hodge structure over $\Lambda$.
The cup-product in singular cohomology yields a natural \emph{cup-product}:
\[ \cup : \HH^{i}(K|k,\Lambda(j)) \otimes_{\Lambda} \HH^{i'}(K|k,\Lambda(j')) \to \HH^{i+i'}(K|k,\Lambda(j+j')) \]
which makes $\HH^{*}(K|k,\Lambda(*)) := \bigoplus_{i \geq 0} \HH^{i}(K|k,\Lambda(i))$ into a graded-commutative $\Lambda$-algebra.

For a \emph{smooth} model $X$ of $K|k$ and nonempty open $k$-subvariety $U$ of $X$, the morphism
\[ \HH^{1}(X,\Lambda(1)) \to \HH^{1}(U,\Lambda(1)), \]
induced by the inclusion $U \hookrightarrow X$, is known to be injective.
Thus the structure map $\HH^{1}(X,\Lambda(1)) \to \HH^{1}(K|k,\Lambda(1))$ is injective as well.
In other words, the underlying $\Lambda$-module of $\HH^{1}(K|k,\Lambda(1))$ can be considered as an \emph{inductive union} of $\HH^{1}(U,\Lambda(1))$ as $U$ varies over the smooth models of $K|k$.
We will tacitly identify $\HH^{1}(U,\Lambda(1))$ with its image in $\HH^{1}(K|k,\Lambda(1))$ whenever $U$ is such a smooth model of $K|k$.

\subsection{Functoriality}

Let $\iota : L \hookrightarrow K$ be a $k$-embedding of function fields over $k$.
By a \emph{model} of $\iota$, we mean a dominant morphism $f : X \to Y$ where $X$ is a model of $K|k$ and $Y$ is a model of $L|k$ which induces $\iota$ on the level of function fields.
Given such a model $f : X \to Y$ of $\iota$, we obtain a canonical map
\[ \iota_{*} : \HH^{i}(L|k,\Lambda(j)) = \varinjlim_{U} \HH^{i}(U,\Lambda(j)) \xrightarrow{f^{*}} \varinjlim_{U} \HH^{i}(f^{-1}(U),\Lambda(j)) \to \HH^{i}(K|k,\Lambda(j)), \]
where $U$ varies over the nonempty open $k$-subvarieties of $Y$.
This morphism does not depend on the choice of model $f$, and this construction makes $\HH^{i}(K|k,\Lambda(j))$ (covariantly) functorial in $K$ with respect to $k$-embeddings.

\subsection{Kummer theory}\label{subsection:kummer-theory}

It is well-known that one has $\HH^{1}(\Gbb_{m},\Lambda(1)) = \Lambda(0)$, as mixed Hodge structures.
Let $K|k$ be a function field and $f \in K^{\times}$ be given.
Let $X$ be a model of $K|k$ and $U$ an open $k$-subvariety of $X$ such that $f \in \Ocal^{\times}(U)$.
Then $f$ corresponds to a morphism $f : U \to \Gbb_{m}$, and hence induces a canonical map of $\Lambda$-modules
\[ \Lambda = \HH^{1}(\Gbb_{m},\Lambda(1)) \xrightarrow{f^{*}} \HH^{1}(U,\Lambda(1)) \to \HH^{1}(K|k,\Lambda(1)). \]
We write $\kappa_{U}(f)$ for the image of $1 \in \Lambda$ in $\HH^{1}(U,\Lambda(1))$ with respect to this map.
It is a straightforward consequence of the K\"unneth formula that the corresponding map
\[ \kappa_U : \Ocal^\times(U) \rightarrow \HH^1(U,\Lambda(1)) \]
is a homomorphism of abelian groups.
As any element of $K^{\times}$ is contained in $\Ocal^{\times}(U)$ for any sufficiently small $U$ as above, we obtain a map
\[ \kappa_{K} : K^{\times} \to \HH^{1}(K|k,\Lambda(1)) \]
by taking the colimit of the maps $\kappa_{U}$.
This map $\kappa_{K}$ is therefore also a morphism of abelian groups.
It is a straightforward consequence of the definitions that $k^{\times}$ is contained in the kernel of $\kappa_{U}$ and $\kappa_{K}$.

Throughout the paper we will write $\Kscr_{\Lambda}(K|k) := (K^{\times}/k^{\times}) \otimes_{\Zbb} \Lambda$.
For $t \in K^{\times}$, we write $t^{\circ}$ for the image of $t$ in $\KL(K|k)$.
We will always use additive notation for the $\Lambda$-module $\KL(K|k)$, while $K^{\times}/k^{\times}$ will be written multiplicatively.
The assignment $K \mapsto \KL(K|k)$ is clearly functorial in $K$ with respect to $k$-embeddings.
For a $k$-embedding $\iota : L \to K$, we will write $\iota_{*} : \KL(L|k) \to \KL(K|k)$ for the corresponding morphism of $\Lambda$-modules.
The Kummer map $\kappa_{K}$ defined above induces a morphism of $\Lambda$-modules
\[ \kappa_{K}^{\Lambda} : \KL(K|k) \to \HH^{1}(K|k,\Lambda(1)), \]
which is natural in $K$ with respect to $k$-embeddings.

\subsection{Milnor K-theory}\label{subsection:milnorK}

The \emph{Milnor K-ring} of $K$ is the graded-commutative ring which is denoted and defined as follows:
\[ \KM_{*}(K) := \frac{\Trm_{*}(K^{\times})}{ \langle  x \otimes (1-x) \ | \ x \in K \smin \{0,1\}\rangle }. \]
Here $\Trm_{*}(K^{\times})$ denotes the (graded) tensor algebra of $K^{\times}$, considered as an abelian group.
It is customary to write $\{f_{1},\ldots,f_{n}\} \in \KM_{n}(K)$ for the product of $f_{1},\ldots,f_{n} \in K^{\times} = \KM_{1}(K)$ in this ring.

Since $\HH^{2}(\Pbb^{1} \smin \{0,1,\infty\}, \Lambda(2)) = 0$, it follows from functoriality that $\kappa_{K}(t) \cup \kappa_{K}(1-t) = 0$ in $\HH^{2}(K|k,\Lambda(2))$ for all $t \in K \smin \{0,1\}$.
Hence the universal property of $\KM_{*}(K)$ shows that $\kappa_{K}$ extends to a morphism of graded-commutative rings
\[ \kappa_{K}^{*} : \KM_{*}(K) \to \HH^{*}(K|k,\Lambda(*)). \]
The $r$-th component of this map, denoted $\kappa_{K}^{r} : \KM_{r}(K) \to \HH^{r}(K|k,\Lambda(r))$, is uniquely determined by the rule $\kappa_{K}^{r}\{f_{1},\ldots,f_{r}\} := \kappa_{K}(f_{1}) \cup \cdots \cup \kappa_{K}(f_{r})$ for $f_{1},\ldots,f_{r} \in K^{\times}$.

\subsection{Residues}\label{subsection:residues}

Suppose that $X$ is a smooth $k$-variety and $Z$ is a smooth closed $k$-subvariety of $X$ which is pure of codimension $1$.
Put $U := X \smin Z$.
Recall that one has a so-called \emph{residue morphism} associated to $(X,Z)$:
\[ \partial_{X,Z} : \HH^{i+1}(U,\Lambda(j+1)) \to \HH^{i}(Z,\Lambda(j)). \]
The following fact seems to be well-known.
For a purely algebraic proof of this assertion, which works with any suitable cohomology theory, we refer the reader to~\cite{MR2431508}*{Proposition 2.6.5} and the surrounding discussion.
\begin{fact}\label{fact:residue-calc}
  In the above context, assume furthermore that $f : X \to \Abb^{1}$ is a morphism of $k$-varieties such that $Z$ is the fibre of $f$ above $0$.
  Consider its restriction $f : U \to \Gbb_{m}$ and the associated function $f \in \Ocal^{\times}(U)$.
  Let $\alpha \in \HH^i(X,\Lambda(j))$ be given, and write $\alpha|_{U}$ for its image in $\HH^{i}(U,\Lambda(j))$ and $\alpha|_{Z}$ for its image in $\HH^{i}(Z,\Lambda(j))$.
  Then one has $\partial_{X,Z}(\kappa_{U}(f) \cup \alpha|_{U}) = \alpha|_{Z}$.
\end{fact}

\subsection{Divisorial valuations}
Suppose that that $v$ is a \emph{divisorial valuation} of $K|k$, i.e.~$v$ arises from a prime divisor on some model of $K|k$.
Equivalently, $v$ satisfies the following properties:
\begin{enumerate}
  \item The value group $vK$ of $v$ is isomorphic (as an ordered abelian group) to $\Zbb$.
        Since $k$ is algebraically closed, this automatically implies that $v$ is trivial on $k$.
  \item The residue field $Kv$ of $v$ is finitely-generated over $k$, and it has transcendence degree $\trdeg(K|k)-1$ over $k$.
\end{enumerate}
In addition to the notations $vK$ for the value group and $Kv$ for the residue field, we will write $\Ocal_{v}$ for the valuation ring, $\mfrak_{v}$ for the valuation ideal, $\Urm_{v} := \Ocal_{v}^{\times}$ for the $v$-units and $\Urm_{v}^{1} := 1 + \mfrak_{v}$ for the principal $v$-units.

Let $X$ be a model of $K|k$.
We say that $X$ is a \emph{$v$-model} provided that the following conditions hold true:
\begin{enumerate}
  \item The valuation $v$ has a (necessarily unique) centre $\xi_{v}$ on $X$.
  \item The centre $\xi_{v}$ is a regular codimension one point in $X$.
\end{enumerate}
Given a $v$-model $X$ of $K|k$ with $v$-centre $\xi_{v}$, we write $X_{v}$ for the closure of $\xi_{v}$ in $X$.
An open $k$-subvariety of $X$ will be called \emph{$v$-open} provided that $\xi_{v} \in U$, or equivalently that $U \cap X_{v}$ is dense in $X_{v}$.
Note that any $v$-open $k$-subvariety $U$ of a $v$-model $X$ of $K|k$ is again a $v$-model, and one has $U \cap X_{v} = U_{v}$.
If $X$ is any $v$-model of $K|k$, we define
\[ \HH^{i}(\Ocal_{v}|k,\Lambda(j)) := \varinjlim_{U} \HH^{i}(U,\Lambda(j)) \]
where $U$ varies over the $v$-open $k$-subvarieties of $X$.
This clearly does not depend on the choice of $v$-model $X$.

\subsection{Residues for divisorial valuations}

Let $v$ be a divisorial valuation of $K|k$ and $X$ a $v$-model of $K|k$.
The restriction maps
\[ \HH^{i}(X,\Lambda(j)) \to \HH^{i}(X \smin X_{v},\Lambda(j)), \ \HH^{i}(X,\Lambda(j)) \to \HH^{j}(X_{v},\Lambda(j)) \]
are compatible with passage to $v$-open $k$-subvarieties of $X$.
Passing to the colimit over the $v$-open $k$-subvarieties of $X$, we obtain canonical maps
\[ \ufrak_{v} : \HH^{i}(\Ocal_{v}|k,\Lambda(j)) \to \HH^{i}(K|k,\Lambda(j)), \ \sfrak_{v} : \HH^{i}(\Ocal_{v}|k,\Lambda(j)) \to \HH^{i}(Kv|k,\Lambda(j)). \]
Similarly, the residue maps $\partial_{X,X_{v}}$ are compatible with passage to $v$-open $k$-subvarieties of $X$, and, passing to the colimit, we thereby obtain a \emph{residue map} associated to $v$:
\[ \partial_{v} : \HH^{i+1}(K|k,\Lambda(j+1)) \to \HH^{i}(Kv|k,\Lambda(j)). \]
Fact~\ref{fact:residue-calc} then translates, in this birational context, to the following.
\begin{fact}\label{fact:birational-residue-calc}
  Let $v$ be a divisorial valuation of $K|k$ and $\varpi \in K^{\times}$ a uniformizer of $v$.
  Let $\alpha \in \HH^{i}(\Ocal_{v}|k,\Lambda(j))$ be given.
  Then one has
  \[ \partial_{v}(\kappa_{K}(\varpi) \cup \ufrak_{v}\alpha) = \sfrak_{v} \alpha, \]
  as elements of $\HH^{i}(Kv|k,\Lambda(j))$.
\end{fact}
\begin{proof}
  This follows from Fact~\ref{fact:residue-calc} since we can find some $v$-model $X$ of $K|k$ such that $\varpi \in \Ocal(X)$ and $X_{v}$ is the zero-locus of $\varpi$.
\end{proof}

To put Fact~\ref{fact:birational-residue-calc} in the right perspective, recall the existence of the \emph{tame-symbol} in Milnor K-theory associated to a divisorial valuation $v$ of $K|k$.
This is a morphism
\[ \partial_{v} : \KM_{r+1}(K) \to \KM_{r}(Kv) \]
which is uniquely characterized by the formula:
\[ \partial_{v} \{\varpi,u_{1},\ldots,u_{r}\} = \{\bar u_{1}, \ldots, \bar u_{r}\}, \]
where $\varpi$ is a uniformizer of $v$, $u_{1},\ldots,u_{r} \in \Urm_{v}$, and $\bar u_{i}$ denotes the image of $u_{i}$ in $Kv^{\times}$.
The residue maps described above are compatible with these tame symbols in the sense that the following diagram commutes:
\begin{equation}
  \begin{tikzcd}
    \KM_{r+1}(K) \ar[r,"\partial_{v}"] \ar[d,"\kappa_{K}^{r+1}"'] & \KM_{r}(Kv) \ar[d,"\kappa_{Kv}^{r}"] \\
    \HH^{r+1}(K|k,\Lambda(r+1)) \ar[r,"\partial_{v}"'] & \HH^{r}(Kv|k,\Lambda(r)).
  \end{tikzcd}
\end{equation}

\section{Algebraic dependence}\label{section:alg-dep}

In this section we discuss the relationship between the cohomological structures described above and algebraic (in)dependence in function fields.
Throughout this section we work with a fixed function field $K|k$.

\subsection{Chomological dimension}\label{subsec:cohom-dim}

Recall that the \emph{Andreotti-Frankel Theorem}~\cite{MR0177422} combined with the universal coefficient theorem shows that $\HH^i(X,\Lambda(j)) = 0$ whenever $X$ is a smooth affine $k$-variety of dimension smaller than $i$.
We immediately obtain the following consequence.

\begin{fact}\label{fact:cohom-dim}
  One has $\HH^i(K|k,\Lambda(j)) = 0$ for all $i > \trdeg(K|k)$.
\end{fact}

This bound on the cohomological dimension of $K|k$ is sharp, as the following lemma shows.
\begin{lemma}\label{lemma:cup-alg-indep}
  Let $f_{1},\ldots,f_{r} \in K^{\times}$ be given.
  The following are equivalent:
  \begin{enumerate}
    \item The element $\kappa_{K}^{r}\{f_{1},\ldots,f_{r}\} \in \HH^{r}(K|k,\Lambda(r))$ vanishes.
    \item The element $\kappa_{K}^{r}\{f_{1},\ldots,f_{r}\} \in \HH^{r}(K|k,\Lambda(r))$ is $\Lambda$-torsion.
    \item The elements $f_{1},\ldots,f_{r} \in K^{\times}$ are algebraically dependent over $k$.
  \end{enumerate}
\end{lemma}
\begin{proof}
  The implication $(3) \Rightarrow (1)$ follows from Fact~\ref{fact:cohom-dim} while $(1) \Rightarrow (2)$ is tautological.
  To conclude, assume that $f_{1},\ldots,f_{r} \in K^{\times}$ are algebraically indepdendent over $k$.
  We will show that $\kappa_{K}^{r}\{f_{1},\ldots,f_{r}\}$ is non-$\Lambda$-torsion in $\HH^{r}(K|k,\Lambda(r))$.
  We proceed by induction on $r$ with $r = 0$ being trivial.
  For the inductive case, choose a divisorial valuation $v$ of $K|k$ which has the following properties:
  \begin{enumerate}
    \item One has $v(f_{1}) \neq 0 = v(f_{2}) = \cdots = v(f_{r})$.
    \item Letting $\bar f_{i}$, $i = 2,\ldots,r$, denote the image of $f_{i}$ in $Kv$, the elements $\bar f_{2},\ldots,\bar f_{r}$ are algebraically indepdendent in $Kv|k$.
  \end{enumerate}
  We then have
  \[ \partial_{v}(\kappa_{K}^{r}\{f_{1},\ldots,f_{r}\}) = v(f_{1}) \cdot \kappa_{Kv}^{r-1}\{\bar f_{2},\ldots,\bar f_{r}\}. \]
  Since $\Lambda$ is a subring of $\Qbb$, this element is non-$\Lambda$-torsion by our inductive hypothesis, and thus the same holds for $\kappa_{K}^{r}\{f_{1},\ldots,f_{r}\}$.
\end{proof}

\subsection{Good models}

Let $K|k$ be a function field and $L$ a subextension of $K|k$ which is relatively algebraically closed in $K$.
We say that a model $X \to B$ of $\iota : L \hookrightarrow K$ is \emph{good} provided that:
\begin{enumerate}
  \item The $k$-varieties $X$, $B$, and the morphism $X \to B$ are all smooth.
  \item The fibres of $X \to B$ are all geometrically integral.
  \item The induced map $X^{\an} \to B^{\an}$ on topological spaces is a fibre bundle.
\end{enumerate}
Clearly, if $f : X \to B$ is a good model, and $U$ is a nonempty $k$-open subvariety in $B$, then the restriction $f^{-1}(U) \to U$ is again good.

The good models of $\iota : L \hookrightarrow K$ are cofinal among all models.
Indeed, if $f : X \to B$ is any model of $\iota$, then $f$ has generically geometrically integral fibres because $L$ was assumed to be relatively algebraically closed in $K$.
Hence the models of $\iota$ satisfying (1) and (2) above are cofinal among all models.
Moreover, if $f : X \to B$ is any model of $\iota$ satisfying (1) and (2), then there exists a nonempty open $k$-subvariety $U$ of $B$ such that the restriction $f^{-1}(U) \to U$ is a good model (see~\cite{MR481096}*{Corollaire 5.1}).

\subsection{Geometric submodules}\label{subsection:geom-submod}

Let $L$ be a subextension of $K|k$ which is algebraically closed in $K$.
In this subsection we study the map induced by the inclusion $L \hookrightarrow K$ on $\HH^{1}$.

\begin{lemma}\label{lemma:geom-submod-incl}
  Let $L$ be a subextension of $K|k$ which is relatively algebraically closed in $K$.
  Then the canonical map
  \[ \HH^{1}(L|k,\Lambda(1)) \to \HH^{1}(K|k,\Lambda(1)) \]
  is injective.
\end{lemma}
\begin{proof}
  Suppose $\alpha$ is in the kernel of this map, and let $X \to B$ be a good model of $L \hookrightarrow K$ such that $\alpha \in \HH^{1}(B,\Lambda(1))$.
  Since $X^{\an} \to B^{\an}$ is a fibre bundle, the map
  \[ \HH^{1}(B,\Lambda(1)) \to \HH^{1}(X,\Lambda(1)) \]
  is injective.
  Since $X$ is smooth, the map $\HH^{1}(X,\Lambda(1)) \to \HH^{1}(K|k,\Lambda(1))$ is injective as well.
  Hence $\alpha = 0$.
\end{proof}

\begin{proposition}\label{proposition:geom-submod-cup}
  Let $L$ be a subextension of $K|k$ which is relatively algebraically closed in $K$ and let $\alpha \in \HH^{1}(K|k,\Lambda(1))$ be given.
  Assume that $\alpha$ is not contained in the image of the injective map
  \[ \HH^{1}(L|k,\Lambda(1)) \to \HH^{1}(K|k,\Lambda(1)). \]
  Then there exists a smooth model $B$ of $L|k$ such that for all closed points $b \in B$ and all systems of regular parameters $(f_{1},\ldots,f_{r})$ of $\Ocal_{B,b}$, the element $\kappa_{K}^{r}\{f_{1},\ldots,f_{r}\} \cup \alpha$ is non-$\Lambda$-torsion (in particular, nontrivial) in $\HH^{r+1}(K|k,\Lambda(r+1))$.
\end{proposition}
\begin{proof}
  Choose a good model $f : X \to B$ of $L \hookrightarrow K$ with $\alpha \in \HH^{1}(X,\Lambda(1))$.
  Let $b \in B$ be a closed point and $f_{1},\ldots,f_{r}$ a system of regular parameters at $b \in B$.
  Let $\xi$ denote the generic point of $Z := f^{-1}(b)$, and note that $f_{1},\ldots,f_{r}$ are a system of regular parameters at $\xi \in X$.
  Replacing $X$ resp.~$B$ with a sufficiently small neighborhood of $\xi$ resp.~$b$, we may also assume that the following conditions hold true:
  \begin{enumerate}
    \item One has $f_{1},\ldots,f_{r} \in \Ocal(B)$.
          Let $W_{i}$ denote the zero-locus of $(f_{1},\ldots,f_{i})$ in $B$ and $Z_{i}$ the zero locus of $(f_{1},\ldots,f_{i})$ in $X$.
          Put $W_{0} := B$ and $Z_{0} := X$.
    \item $W_{0} \supsetneq W_{1} \supsetneq \cdots \supsetneq W_{r}$ is a flag of smooth integral closed subvarieties of $W_{0} = B$ with $W_{i+1}$ having codimension one in $W_{i}$.
    \item $Z_{0} \supsetneq Z_{1} \supsetneq \cdots \supsetneq Z_{r}$ is a flag of smooth integral closed subvarieties of $Z_{0} = X$ with $Z_{i+1}$ having codimension one in $Z_{i}$.
    \item For all $i = 0,\ldots,r-1$, the function $f_{i+1}$ is a regular parameter for the generic point of $W_{i+1}$ in $W_{i}$, and similarly for the generic point of $Z_{i+1}$ in $Z_{i}$.
  \end{enumerate}
  Put $\partial_{i} := \partial_{Z_{i-1},Z_{i}}$ for $i = 1,\ldots,r$.
  Applying Fact~\ref{fact:residue-calc} successively $r$ times, we find
  \[ (\partial_{r} \circ  \partial_{r-1} \circ \cdots \circ \partial_{1})(\kappa_{Z_{0} \smin Z_{1}}(f_{1}) \cup \cdots \cup \kappa_{Z_{r-1} \smin Z_{r}}(f_{r}) \cup \alpha) = \beta, \]
  where $\beta$ denotes the image of $\alpha$ in $\HH^{1}(Z,\Lambda(1))$.
  This map to $\HH^{1}(Z,\Lambda(1))$ fits in an exact sequence of the form
  \[ 0 \to \HH^{1}(B,\Lambda(1)) \to \HH^{1}(X,\Lambda(1)) \to \HH^{1}(Z,\Lambda(1)) \]
  since $X^{\an} \to B^{\an}$ is a fibre bundle.
  Passing to the colimit over good models of $L \hookrightarrow K$, we obtain a similar exact sequence
  \[ 0 \to \HH^{1}(L|k,\Lambda(1)) \to \HH^{1}(K|k,\Lambda(1)) \to \HH^{1}(k(Z)|k,\Lambda(1)). \]
  Writing $v_{i}$ for the divisorial valuation on $k(Z_{i-1})$ associated to $Z_{i}$, the above calculation shows that
  \[ (\partial_{v_{r}} \circ \cdots \circ \partial_{v_{1}})(\kappa_{K}^{r}\{f_{1},\ldots,f_{r}\} \cup \alpha) \]
  agrees with the image of $\alpha$ in $\HH^{1}(k(Z)|k,\Lambda(1))$.
  The assertion follows since this element is nontrivial, and $\HH^{1}(k(Z)|k,\Lambda(1))$ is torsion-free.
\end{proof}

\section{Picard $1$-motives}\label{section:picard}

Recall that a \emph{$1$-motive} over $k$ consists of the following data:
\begin{enumerate}
  \item A semiabelian $k$-variety $\Gbf$.
  \item A finitely-generated free abelian group $\Lbf$.
  \item A morphism $\Lbf \to \Gbf(k)$.
\end{enumerate}
This data is summarized as a complex $[\Lbf \to \Gbf]$ of group schemes over $k$ where $\Lbf$ is placed in degree $0$ and $\Gbf$ in degree $-1$.
A morphism of $1$-motives is simply a morphism of such complexes.
The set of morphisms $\Hom_{k}(\Mbf_{1},\Mbf_{2})$ between two $1$-motives over $k$ naturally forms an abelian group.

\subsection{Hodge realizations}\label{subsection:hodge-realization}

Let $\Mbf = [\Lbf \to \Gbf]$ be a $1$-motive over $k$.
The \emph{Hodge realization} of $\Mbf$, denoted $\HH(\Mbf)$, is constructed as follows (see~\cite{MR0498552}*{\S 10}).
First, consider the exponential sequence
\[ 0 \to \HH_{1}(\Gbf^{\an},\Zbb) \to \Lie\Gbf^{\an} \to \Gbf^{\an} \to 0. \]
Pull back this sequence with respect to the map $\Lbf \to \Gbf(k) \hookrightarrow \Gbf^{\an}$ to obtain $\HH(\Mbf)$ which then fits in an exact sequence of the form
\[ 0 \to \HH_{1}(\Gbf^{\an},\Zbb) \to \HH(\Mbf) \to \Lbf \to 0. \]
The mixed Hodge structure of $\HH(\Mbf)$ is the unique one making the exact sequence above compatible with the usual mixed Hodge structure on $\HH_{1}(\Gbf^{\an},\Zbb)$, the trivial Hodge structure on $\Lbf$, where $\Frm^{0}(\HH(\Mbf) \otimes \Cbb)$ is the kernel of the induced map $\HH(\Mbf) \otimes_\Zbb \Cbb \rightarrow \Lie \Gbf^\an$ obtained from the construction of $\HH(\Mbf)$.
We will need the following lemma, which is the simplest case of Deligne's much more general result~\cite{MR0498552}*{10.1.3}.

\begin{lemma}\label{lemma:hodge-realization-iso}
  Let $\Mbf = [\Lbf \to \Abf]$ be a $1$-motive over $k$ with $\Abf$ an abelian variety.
  Consider the $1$-motive $\Zbb := [\Zbb \to 0]$.
  Then the trivial Hodge structure $\Zbb$ is isomorphic to $\HH(\Zbb)$, and $\HH(-)$ induces an isomorphism $\Hom_{k}(\Zbb,\Mbf) \cong \Hom_{\MHS}(\Zbb,\HH(\Mbf))$.
\end{lemma}
\begin{proof}
  The fact that $\Zbb$ (as a mixed Hodge structure) agrees with $\HH(\Zbb)$ follows directly from the definitions.
  View $\Abf$ and $\Lbf$ as the $1$-motives $[0 \to \Abf]$ and $[\Lbf \to 0]$ respectively.
  We have an exact sequence of $1$-motives,
  \[ 0 \to \Abf \to \Mbf \to \Lbf \to 0, \]
  and a corresponding exact sequence of Hodge realizations:
  \[ 0 \rightarrow \HH(\Abf) \rightarrow \HH(\Mbf) \rightarrow \HH(\Lbf) \rightarrow 0. \]
  Note $\HH(\Abf) = \HH_{1}(\Abf^{\an},\Zbb)$ and $\HH(\Lbf) = \Lbf$.

  The lemma is easily verified in the case where $\Mbf = \Abf$ or $\Mbf = \Lbf$.
  Indeed, the yoga of weights ensures $\Hom_{\MHS}(\Zbb,\HH(\Abf)) = 0$, while $\Hom_{k}(\Zbb,\Abf) = 0$ by definition.
  On the other hand $\Hom_{k}(\Zbb,\Lbf) = \Lbf$ while $\Hom_{\MHS}(\Zbb,\Lbf) = \Lbf$ and the map in question is the obvious isomorphism.

  A simple diagram chase shows it is enough to prove that the map
  \[ \Ext^{1}_{k}(\Zbb,\Abf) \to \Ext^{1}_{\MHS}(\Zbb,\HH(\Abf)) \]
  induced by $\HH(-)$ is injective, where the $\Ext^{1}_{k}$ refers to extensions of $1$-motives over $k$.
  One has $\Ext^{1}_{k}(\Zbb,\Abf) = \Abf(k)$, merely as a consequence of the definitions.
  Also, it is well-known that one has $\Ext^{1}_{\MHS}(\Zbb,\HH(\Abf)) = \Abf^{\an}$ (see~\cite{MR605338}), and the corresponding map from $\Abf(k)$ to $\Abf^{\an} = \Abf(\Cbb)$ is the usual (injective) inclusion.
  This concludes the proof of the lemma.
\end{proof}

\subsection{Picard 1-motives}\label{subsection:picard-1-motives}

Let $X$ be a smooth projective integral $k$-variety, and $U$ a nonempty open $k$-subvariety of $X$.
Put $Z := X \smin U$.
Consider $\Div^{0}(X)$, the group of algebraically trivial Weil-divisors on $X$, as well as the subgroup $\Div^{0}_{Z}(X)$ of algebraically trivial Weil divisors on $X$ which are supported on $Z$.
Note $\Div^{0}_{Z}(X)$ is a finitely-generated free abelian group.

Next, consider the Picard variety $\Picbf^{0}_{X}$ of $X$.
This is an abelian $k$-variety, and we have a canonical morphism
\[ \Div^{0}_{Z}(X) \hookrightarrow \Div^{0}(X) \to \Pic^{0}(X) = \Picbf^{0}_{X}(k), \]
mapping a Weil divisor to its associated line bundle.
We thereby obtain the so-called \emph{Picard $1$-motive} of $U$ (associated to the inclusion $U \hookrightarrow X$):
\[ \Mbf^{1,1}(U) := [\Div^{0}_{Z}(X) \to \Picbf^{0}_{X}]. \]
Whenever $V \subset U$ is a nonempty open $k$-subvariety, we obtain a canonical morphism
\[ \Mbf^{1,1}(U) \rightarrow \Mbf^{1,1}(V) \]
of $1$-motives over $k$, which simply arises from the inclusion $\Div^{0}_{X \smin U}(X) \hookrightarrow \Div^{0}_{X \smin V}(X)$.

The following theorem, due to Barbieri-Viale, Srinivas~\cite{MR1891270} and Ramachandran~\cite{MR1837235} describe the Hodge realization of such Picard $1$-motives.
\begin{theorem}[\cite{MR1891270}*{Theorem 4.7},~\cite{MR1837235}*{Theorem 2.5}]\label{theorem:picard/picard-hodge}
  In the above context, there is a canonical isomorphism of mixed Hodge structures $\HH(\Mbf^{1,1}(U)) \cong \HH^1(U,\Zbb(1))$ which is functorial with respect to open embeddings $V \hookrightarrow U$ of open $k$-subvarieties of $X$.
\end{theorem}

\section{An anabelian result}\label{section:anab}

In this section, we discuss an anabelian result from which we will eventually deduce our main theorem.
Write $\Lambda_{\neq 0}$ for the set of nonzero elements of $\Lambda$.
Since $\Lambda$ is a subring of $\Qbb$, we see that for every $x \in \KL(K|k)$, there exists some $t \in K^{\times}$ such that $t^{\circ} \in \Lambda_{\neq 0} \cdot x$.
Given two elements $x,y \in \KL(K|k)$ and elements $u,v \in K^{\times}$ such that $u^{\circ} \in \Lambda_{\neq 0} \cdot x$ and $v^{\circ} \in \Lambda_{\neq 0} \cdot y$, we say that $x,y$ are \emph{(in)dependent} provided that $u,v$ are algebraically (in)dependent over $k$.
It is easy to see that this definition does not depend on the choice of $u,v$ as above, and that $x,y$ are dependent if and only if they are not independent (this again relies on the assumption that $\Lambda \subset \Qbb$).

For a subextension $M$ of $K|k$, the canonical map
\[ \KL(M|k) \to \KL(K|k) \]
is injective since $\Lambda$ is flat over $\Zbb$.
We will tacitly identify $\KL(M|k)$ with its image in $\KL(K|k)$ with respect to this inclusion.
For a subset $S$ of $K$, we write
\[ \acl_{K}(S) := \bar{k(S)} \cap K \]
for the relative algebraic closure of $k(S)$ in $K$.
We say that a submodule $\Kscr$ of $\KL(K|k)$ is \emph{rational} provided that there exists some $t \in K \smin k$ such that $\acl_{K}(t) = k(t)$ and such that $\Kscr = \KL(k(t)|k)$.

Next suppose that $L|l$ is a further function field over an algebraically closed field $l$ of characteristic $0$, and let
\[ \phi : \KL(K|k) \cong \KL(L|l) \]
be an isomorphism of $\Lambda$-modules.
We say that
\begin{enumerate}
  \item \emph{$\phi$ is compatible with $\acl$} provided that for all $x,y \in \KL(K|k)$, the pair $x,y$ is dependent in $\KL(K|k)$ if and only if the pair $\phi x, \phi y$ is dependent in $\KL(L|l)$.
  \item \emph{$\phi$ is compatible with rational submodules} provided that $\phi$ induces a bijection on rational submodules of $\KL(K|k)$ resp.~$\KL(L|l)$.
\end{enumerate}
The collection of all isomorphisms $\KL(K|k) \cong \KL(L|l)$ which are compatible with $\acl$ and with rational submodules will be denoted by
\[ \Isom^{\acl}_{\rat}(\KL(K|k),\KL(L|l)). \]
Note that for any $\phi \in \Isom^{\acl}_{\rat}(\KL(K|k),\KL(L|l))$ and $\epsilon \in \Lambda^{\times}$, the corresponding isomorphism $\epsilon \cdot \phi$ is again compatible with $\acl$ and with rational submodules.
We thus obtain an action of $\Lambda^{\times}$ on $\Isom^{\acl}_{\rat}(\KL(K|k),\KL(L|l))$, and we denote its orbits by
\[ \Isommod^{\acl}_{\rat}(\KL(K|k),\KL(L|l)). \]
Any isomorphism of fields $K \cong L$ restricts to an isomorphism $k \cong l$, hence we obtain a canonical map
\[ \Isom(K,L) \to \Isom^{\acl}_{\rat}(\KL(K|k),\KL(L|l)) \twoheadrightarrow \Isommod^{\acl}_{\rat}(\KL(K|k),\KL(L|l)), \]
which is the focus of the following main result of this section.

\begin{theorem}\label{theorem:main-anab}
  In the above context, assume furthermore that $\trdeg(K|k) \geq 2$.
  Then the canonical map
  \[ \Isom(K,L) \rightarrow \Isommod^{\acl}_{\rat}(\KL(K|k),\KL(L|l)) \]
  is a bijection.
\end{theorem}

We stated Theorem~\ref{theorem:main-anab} as a theorem because it may be of independent interest.
However, it can be deduced from known results in the literature in certain special cases.
For example, if $\Lambda = \Zbb$, this theorem follows from the main results of Bogomolov-Tschinkel~\cite{MR2537087} and/or Cadoret-Pirutka~\cite{CadoretPirutka}.
If $\Lambda$ is a \emph{proper} subring of $\Qbb$, then this theorem can be deduced from the work of Pop~\cite{MR2867932}.
And finally, if $\trdeg(K|k) \geq 5$, then one can deduce this result from the work of Evans-Hrushovski~\cites{MR1078211,MR1356137} and Gismatullin~\cite{MR2439644}, along with arguments similar to the ones appearing below.
Moreover, in all of these cases the condition of \emph{compatibility with rational submodules} can be removed.

In this respect, the most interesting case of Theorem~\ref{theorem:main-anab} is where $\Lambda = \Qbb$, and where one considers function fields of transcendence degree $\geq 2$.
In such cases, we do not know of a straightforward way to deduce this result from what has appeared in the literature.

The rest of this section is devoted to proving Theorem~\ref{theorem:main-anab}, and the bulk of the proof is devoted to constructing a (functorial) left inverse of the map appearing in its statement.
For the rest of this section, we put ourselves in the context of Theorem~\ref{theorem:main-anab} and fix an element $\phi \in \Isom^{\acl}_{\rat}(\KL(K|k),\KL(L|l))$.

\subsection{Divisorial valuations}\label{subsection:divisorial-valuations}

For a divisorial valuation $v$ of $K|k$, we write
\[ \Uscr_{v} := \image((\Urm_{v}/k^{\times}) \otimes_{\Zbb} \Lambda \to \KL(K|k)), \ \Uscr_{v}^{1} := \image((\Urm_{v}^{1} \cdot k^{\times}/k^{\times}) \otimes_{\Zbb} \Lambda \to \KL(K|k)). \]
The maps in the formula above are the ones induced by the inclusions $\Urm_{v} \hookrightarrow K^{\times}$ and $\Urm_{v}^{1} \hookrightarrow K^{\times}$, respectively.
Note that $\Uscr_{v}^{1} \subset \Uscr_{v} \subset \KL(K|k)$, and that the map $\Urm_{v} \to {(Kv)}^{\times}$ induces an isomorphism $\Uscr_{v}/\Uscr_{v}^{1} \cong \KL(Kv|k)$.

We will need to use a variant of the \emph{local theory} from almost-abelian anabelian geometry in order to show the compatibility of $\phi$ with divisorial valuations.
Such local theories have been extensively developed, see~\cites{MR1977585,MR2735055,TopazCrelle,MR3552293} for instance.
However, the precise statement we need in our context has not appeared in the literature.
For the sake of completeness, we provide the details for the local theory needed here in an appendix to this paper.
We summarize the precise result we need in the following fact, which follows directly from Theorem~\ref{theorem:localtheory} in the appendix.

\begin{fact}\label{fact:anab/local-theory}
  In the above context, for all divisorial valuations $v$ of $K|k$, there exists a unique divisorial valuation $v^\phi$ of $L|l$ such that $\phi(\Uscr_v) = \Uscr_{v^\phi}$ and $\phi(\Uscr_v^1) = \Uscr_{v^\phi}^1$.
\end{fact}

\subsection{Rational submodules}\label{section:rational-submodules}

For $t \in K \smin k$, put
\[ \Kscr_{t} := \KL(\acl_{K}(t)|k) \]
considered as a submodule of $\KL(K|k)$.
\begin{lemma}\label{lemma:dim-one-geom}
  A submodule $\Kscr$ of $\KL(K|k)$ has the form $\Kscr_{t}$ if and only if it satisfies the following conditions:
  \begin{enumerate}
    \item The submodule $\Kscr$ is nontrivial.
    \item For all $\alpha \in \Kscr$ and $\beta \in \KL(K|k)$ dependent with $\alpha$, one has $\beta \in \Kscr$.
    \item The submodule $\Kscr$ is minimal among submodules of $\KL(K|k)$ satisfying (1) and (2).
  \end{enumerate}
\end{lemma}
\begin{proof}
  Note that a submodule $\Kscr$ has the form $\KL(M|k)$ for some relatively algebraically closed subextension $M$ of $K|k$ if and only if it satisfies condition (2).
  The assertion follows easily from this.
\end{proof}

We say that $t$ is \emph{general in $K|k$} provided that $K$ is regular over $k(t)$.
Note that if $t$ is general in $K|k$ then $\Kscr_{t}$ is a rational submodule of $\KL(K|k)$.
We will need the following \emph{birational Bertini} result.

\begin{fact}[Birational Bertini~\cite{MR0344244}*{Ch. VIII, pg. 213}]\label{fact:anab/birational-bertini}
  Let $x,y \in K$ be algebraically independent over $k$.
  For all but finitely many $a \in k$, the element $x+a \cdot y$ is general in $K|k$.
\end{fact}

\subsection{Divisors on curves}

Let $t \in K \smin k$ be given, and put $\Kscr := \Kscr_{t}$.
Consider the following collection of submodules of $\Kscr$:
\[ \Dscr_{t} = \Dscr_{\Kscr} := \{ \Uscr_{v} \cap \Kscr \ | \ \Kscr \not\subset \Uscr_{v} \} \]
where $v$ varies over the divisorial valuations of $K|k$.
Also write $\Dbb_{t} = \Dbb_{\acl_{K}(t)}$ for the collection of divisorial valuations of $\acl_{K}(t)|k$, so that $\Dbb_{t}$ is in bijection with the closed points of the unique projective normal model of $\acl_{K}(t)|k$.

\begin{lemma}\label{lemma:curve-divisor-bijection}
  In the above context, the following hold:
  \begin{enumerate}
    \item For all $\Uscr \in \Dscr_{t}$, the quotient $\Kscr/\Uscr$ is isomorphic to $\Lambda$.
    \item One has a canonical bijection $\Dbb_{t} \cong \Dscr_{t}$ defined by $w \mapsto \Uscr_{w}$, $w \in \Dbb_{t}$ with inverse sending $\Uscr = \Uscr_{v} \cap \Kscr$ to the restriction of $v$ to $\acl_{K}(t)$.
  \end{enumerate}
\end{lemma}
\begin{proof}
  For (1), let $v$ be a divisorial valuation of $K|k$ with $\Kscr \not\subset \Uscr_{v}$ and put $\Uscr = \Uscr_{v} \cap \Kscr$.
  Note that $\KL(K|k)/\Uscr_{v}$ is isomorphic to $\Lambda$, hence one has a canonical injective morphism of $\Lambda$-modules
  \[ \Kscr/\Uscr \hookrightarrow \KL(K|k)/\Uscr_{v} \cong \Lambda. \]
  The image of this map is nontrivial by assumption and since $\Lambda$ is a subring of $\Qbb$, hence a PID, assertion (1) follows.

  For (2), put $M := \acl_{K}(t)$ and suppose first that $w$ is a divisorial valuation of $M|k$.
  Then there exists a divisorial valuation $v$ of $K|k$ whose restriction to $M$ is $w$, hence $\Uscr_{w} \subset \Uscr_{v} \cap \Kscr$ while $\Kscr \not\subset \Uscr_{v}$.
  As both $\Kscr/\Uscr_{w}$ and $\Kscr/\Uscr_{v} \cap \Kscr$ are isomorphic to $\Lambda$ and $\Kscr/\Uscr_{v} \cap \Kscr$ is a quotient of $\Kscr/\Uscr_{w}$, it follows that $\Uscr_{w} = \Uscr_{v} \cap \Kscr \in \Dscr_{t}$.

  Similarly if $\Uscr$ is an element of $\Dscr_{t}$, and $v$ is a divisorial valuation of $K|k$ with $\Uscr = \Uscr_{v} \cap \Kscr$ and $\Kscr \not\subset \Uscr_{v}$, then the restriction $w$ of $v$ to $M$ is a divisorial valuation of $M|k$ with $\Uscr_{w} \subset \Uscr_{v} \cap \Kscr$.
  Arguing as above, we find again that $\Uscr_{w} = \Uscr_{v} \cap \Kscr = \Uscr$.
\end{proof}

\subsection{Rational-like collections}\label{subsection:rational-like-collections}

Assume now that $t$ is a general element of $K|k$ so that $\Kscr := \Kscr_{t}$ is a rational submodule of $\KL(K|k)$.
By Lemma~\ref{lemma:curve-divisor-bijection}, we have $\Kscr/\Uscr \cong \Lambda$ for every $\Uscr \in \Dscr_{t}$.
Consider a collection of such isomorphisms:
\[ \Phi = {(\Phi_{\Uscr} : \Kscr/\Uscr \xrightarrow{\cong} \Lambda)}_{\Uscr \in \Dscr_{t}}. \]
Any element of $\Kscr$ is contained in all but finitely many of the $\Uscr \in \Dscr_{t}$ by Lemma~\ref{lemma:curve-divisor-bijection}, hence $\Phi$ induces a canonical map
\[ \div_{\Phi} : \Kscr \to \bigoplus_{\Uscr \in \Dscr_{t}} \Lambda \cdot [\Uscr], \ \ \div_{\Phi}(x) = \sum_{\Uscr \in \Dscr_{t}} \Phi_{\Uscr}(x) \cdot [\Uscr]. \]
Here $[\Uscr]$, $\Uscr \in \Dscr_{t}$, denote formal basis elements for the direct sum.

We say that $\Phi$ is a \emph{rational-like collection} provided that $\div_{\Phi}$ fits in a short exact sequence of the form
\[ 0 \to \Kscr \xrightarrow{\div_{\Phi}} \bigoplus_{\Uscr \in \Dscr_{t}} \Lambda \cdot [\Uscr] \xrightarrow{\sum} \Lambda \to 0. \]
If $\Phi$ is a rational-like collection and $\epsilon \in \Lambda^{\times}$ is given, then we obtain another rational-like collection
\[ \epsilon \cdot \Phi := {(\epsilon \cdot \Phi_{\Uscr})}_{\Uscr \in \Dscr_{t}}. \]

By Lemma~\ref{lemma:curve-divisor-bijection}, there is a \emph{canonical} rational-like collection for $\Kscr$, constructed from the field structure of $M := \acl_{K}(t) = k(t)$, as follows.
For $\Uscr \in \Dscr_{t}$ and $w \in \Dbb_{t}$ such that $\Uscr = \Uscr_{w}$, the isomorphism $\Phi^{\can}_{\Uscr}$ is the unique one making the following diagram commute:
\[
  \begin{tikzcd}
    \Kscr \ar[d, two heads] \ar[r,equal] & (M^\times/k^\times) \otimes_\Zbb \Lambda  \ar[r,"w \otimes \Lambda"] & \Zbb \otimes_\Zbb \Lambda = \Lambda \\
    \Kscr/\Uscr \ar[rr,equal] & {} & \Kscr/\Uscr. \ar[u,"\Phi_{\Uscr}^{\can}"']
  \end{tikzcd}
\]

Write $\Phi^{\can}_{\Kscr} := {(\Phi^{\can}_{\Uscr})}_{\Uscr \in \Dscr_{\Kscr}}$.
This is clearly a rational-like collection, which we call \emph{the canonical rational-like collection of $\Kscr$}.
We simplify the notation by writing $\div_{\can} := \div_{\Phi^{\can}_{\Kscr}}$ if $\Kscr$ is understood from context.

\begin{lemma}\label{lemma:rational-like-canon}
  In the above context, let $\Phi$ be a rational-like collection for $\Kscr$.
  Then there exists a unique $\epsilon = \epsilon_{\Kscr} \in \Lambda^{\times}$ such that $\Phi = \epsilon \cdot \Phi^{\can}_{\Kscr}$.
\end{lemma}
\begin{proof}
  By Lemma~\ref{lemma:curve-divisor-bijection}, for each $\Uscr \in \Dscr_{\Kscr}$, we may choose an $\epsilon_{\Uscr} \in \Lambda^{\times}$ such that
  \[ \Phi_{\Uscr} = \epsilon_{\Uscr} \cdot \Phi^{\can}_{\Uscr}. \]
  We must show that $\epsilon_{\Uscr}$ is independent of the choice of $\Uscr$.
  For two different $\Uscr,\Vscr \in \Dscr_{\Kscr}$, there exists a unique $x \in \Kscr$ such that
  \[ \div_{\can}(x) = [\Uscr] - [\Vscr]. \]
  Hence $\div_{\Phi}(x) = \epsilon_{\Uscr} \cdot [\Uscr] - \epsilon_{\Vscr} \cdot [\Vscr]$.
  The ``exactness'' in the definition of a rational-like collection (applied to $\Phi$) shows that $\epsilon_{\Uscr} - \epsilon_{\Vscr} = 0$.
  Hence $\epsilon_{\Uscr}$ does not depend on $\Uscr$.
\end{proof}

\subsection{Rational synchronization}\label{subsection:rational-syncro}

Let $t$ be a general element of $K|k$.
By Lemma~\ref{lemma:curve-divisor-bijection}, $\Dscr_{t}$ is parameterized by $\Pbb^{1}(k) = k \cup \{\infty\}$ by identifying $a \in k$ resp.~$\infty$ with $\Uscr_{w}$ with $w$ the divisorial valuation whose centre is $t = a$ resp.~$t = \infty$.
Write $\Uscr_{t,a}$ for the element of $\Dscr_{t}$ corresponding to $a \in k \cup \infty$.
Note that
\[ \div_{\can}{(t-c)}^{\circ} = [\Uscr_{t,c}] - [\Uscr_{t,\infty}] \]
for all $c \in k$.
Also, if $\Uscr_{1},\Uscr_{2} \in \Dscr_{t}$ are two distinct elements, then there exists a general element $x$ of $K|k$ such that $k(x) = k(t)$ and such that
\[ \div_{\can}x^{\circ} = [\Uscr_{1}] - [\Uscr_{2}]. \]

\begin{proposition}\label{proposition:rational-syncro}
  Let $\phi : \KL(K|k) \cong \KL(L|l)$ be an isomorphism of $\Lambda$-modules which is compatible with $\acl$ and with rational submodules, and let $x$ be a general element of $K|k$.
  Then there exists a general element $y$ of $L|l$, a unit $\epsilon \in \Lambda^{\times}$, and a bijection $\eta : k \cong l$ such that $\eta 0 = 0$, $\eta 1 = 1$, and $\phi{(x-a)}^{\circ} = \epsilon \cdot {(y-\eta a)}^{\circ}$ for all $a \in k$.
\end{proposition}
\begin{proof}
  Put $\Kscr := \Kscr_{x}$ and recall that $\Lscr := \phi \Kscr$ is a rational submodule of $\KL(L|l)$.
  By Fact~\ref{fact:anab/local-theory}, $\phi$ induces a bijection
  \[ \Uscr \mapsto \phi \Uscr \ : \ \Dscr_{\Kscr} \xrightarrow{\cong} \Dscr_{\Lscr}. \]
  Consider the canonical rational-like collection $\Phi := \Phi^{\can}_{\Kscr}$ on $\Kscr$.
  Let $\Psi$ denote the rational-like collection on $\Lscr$ induced by $\Phi$ and $\phi$.
  Explicitly, $\Psi_{\Vscr} : \Lscr/\Vscr \cong \Lambda$, $\Vscr \in \Dscr_{\Lscr}$, is the isomorphism
  \[ \Lscr/\Vscr \xrightarrow{\phi^{-1}} \Kscr/\Uscr \xrightarrow{\Phi_{\Uscr}} \Lambda \]
  where $\Uscr = \phi^{-1}\Vscr$.
  By Lemma~\ref{lemma:rational-like-canon} there exists an $\epsilon \in \Lambda^{\times}$ such that $\Psi = \epsilon^{-1} \cdot \Phi^{\can}_{\Lscr}$, and by the construction of $\Psi$ we have a commutative diagram with exact rows:
  \[
    \begin{tikzcd}
      0 \ar[r] & \Kscr \ar[r,"\div_{\can}"{pos=0.4}] \ar[d,"\phi"'] & \bigoplus_{\Uscr \in \Dscr_\Kscr} \Lambda \cdot [\Uscr] \ar[r,"\mathrm{sum}"] \ar[d,"\phi"] & \Lambda \ar[r] \ar[d,equal] & 0 \\
      0 \ar[r] & \Lscr \ar[r,"\div_{\Psi}"'{pos=0.5}] & \bigoplus_{\Vscr \in \Dscr_\Lscr} \Lambda \cdot [\Vscr] \ar[r,"\mathrm{sum}"'] & \Lambda \ar[r] & 0.
    \end{tikzcd}
  \]
  The $\phi$ in the middle is shorthand for the morphism defined by $[\Uscr] \mapsto [\phi \Uscr]$.

  Note that $\div_{\Phi}x^{\circ} = [\Uscr_{x,0}] - [\Uscr_{x,\infty}]$, hence $\div_{\Psi}(\phi x^{\circ}) = [\phi \Uscr_{x,0}] - [\phi \Uscr_{x,\infty}]$.
  By the discussion above, there exists a general element $y$ of $L|l$ such that $\Kscr_{y} = \Lscr$, and such that
  \[ \phi \Uscr_{x,0} = \Uscr_{y,0}, \ \phi\Uscr_{x,\infty} = \Uscr_{y,\infty}. \]
  Replacing $y$ with an element of the form $c \cdot y$ for some $c \in l^{\times}$ we may assume furthermore that $\phi \Uscr_{x,1} = \Uscr_{y,1}$.
  Define the bijection $\eta : k \cong l$ as the unique one satisfying $\phi \Uscr_{x,a} = \Uscr_{y,\eta a}$, for $a \in k$.
  Then for all $a \in k$, we have
  \begin{align*}
    \div_\can(\epsilon^{-1} \cdot {\phi(x-a)}^\circ) &= \div_\Psi({\phi (x-a)}^\circ) \\
      &= [\phi \Uscr_{x,a}] - [\phi \Uscr_{x,\infty}] \\
      &= [\Uscr_{y,\eta a}] - [\Uscr_{y,\infty}] \\
      &= \div_\can({(y-\eta a)}^\circ)
  \end{align*}
  The injectivity of $\div_\can$ implies that ${\phi (x-a)}^\circ = \epsilon \cdot {(y-\eta a)}^\circ$, as required.
\end{proof}

\subsection{Multiplicative synchronization}\label{subsection:multiplicative-syncro}

At this point, our proof of Theorem~\ref{theorem:main-anab} will use an adaptation of arguments due to Pop~\cite{MR2867932}*{\S6}.
Following Proposition~\ref{proposition:rational-syncro}, we will say that $\phi \in \Isom_{\rat}^{\acl}(\KL(K|k),\KL(L|l))$ is \emph{synchronized} provided that there exists some general element $x$ of $K|k$, some general element $y$ of $L|l$, and some bijection $\eta : k \cong l$ with $\eta 0 = 0$ and $\eta 1 = 1$, such that
\[ \phi {(x-a)}^{\circ} = {(y-\eta a)}^{\circ} \]
for all $a \in k$.
To specify $x$, $y$ and $\eta$, we may say that $\phi$ is \emph{synchronized by $x$ and $y$ via $\eta$}.
By Proposition~\ref{proposition:rational-syncro}, for any $\phi \in \Isom_{\rat}^{\acl}(\KL(K|k),\KL(L|l))$, there exists some $\epsilon \in \Lambda^{\times}$ such that $\epsilon \cdot \phi$ is synchronized.

The canonical map $K^{\times}/k^{\times} \to \KL(K|k)$ is injective, and we will identify $K^{\times}/k^{\times}$ with its image in $\KL(K|k)$, and similarly for $L^{\times}/l^{\times}$.
We show that a synchronized $\phi$ is compatible with these lattices.

\begin{proposition}\label{proposition:multiplicative-syncro}
  Assume that $\phi \in \Isom_{\rat}^{\acl}(\KL(K|k),\KL(L|l))$ is synchronized.
  Then one has $\phi (K^{\times}/k^{\times}) = L^{\times}/l^{\times}$.
\end{proposition}
\begin{proof}
  It suffices to prove that $\phi (K^{\times}/k^{\times}) \subset L^{\times}/l^{\times}$ since $\phi^{-1}$ is also synchronized whenever $\phi$ is.
  Put $\Mscr := \phi^{-1}(L^{\times}/l^{\times}) \cap (K^{\times}/k^{\times})$ and let $M^{\times}$ denote the preimage of $\Mscr$ in $K^{\times}$.
  Our goal is to show that $M := M^{\times} \cup \{0\} = K$.

  Suppose $\phi$ is synchronized by $x$ and $y$ via $\eta$.
  Note that ${k(x)}^{\times} \subset M^{\times}$ since ${k(x)}^{\times}$ is multiplicatively generated by elements of the form $x-a$, $a \in k$.

  More generally, assume that $u \in M^{\times}$ is general in $K|k$.
  By Proposition~\ref{proposition:rational-syncro}, there exists a bijection $\gamma : k \cong l$, a general element $w$ of $L|l$, and an $\epsilon \in \Lambda^\times$, such that $\gamma 0 = 0$, $\gamma 1 = 1$ and
  \[ \phi{(u-a)}^\circ = \epsilon \cdot {(w-\gamma a)}^\circ \]
  for all $a \in k$.
  Note in particular that $\phi u^\circ = \epsilon \cdot w^\circ$, while $\phi u^\circ \in L^\times/l^\times$.

  We claim that $\epsilon \in \Zbb$.
  Write $\epsilon = m/n$ for integers $m,n$ and $n > 0$.
  Since $l^{\times}$ is divisible, the above observations show that there exists $g \in L^{\times}$ such that $w^{m} = g^{n}$, hence $w$ and $g$ are algebraically dependent.
  As $w$ is general in $L|l$, it follows that $g \in k(w)$ and comparing $w$-adic valuations we find $\epsilon = m/n \in \Zbb$.

  With $\epsilon \in \Zbb$ as above, one has
  \[ \phi {(u-a)}^{\circ} = \epsilon \cdot {(w-\gamma a)}^{\circ} = {({(w-\eta a)}^{\epsilon})}^{\circ} \]
  for all $a \in k$, hence $u - a \in M^{\times}$.
  As ${k(u)}^{\times}$ is multiplicatively generated by $u-a$ for $a \in k$, we deduce that ${k(u)}^{\times}$ is contained in $M^{\times}$.

  Finally, since $\Lambda \subset \Qbb$, we see that for all $t \in K^{\times}$ there exists some integer $n > 0$ such that $n \cdot \phi t^{\circ} \in L^{\times}/l^{\times}$.
  In other words, $t^{n} \in M^{\times}$, so that $K^{\times}/M^{\times}$ is torsion.
  Bringing together the above observations:
  \begin{enumerate}
    \item The quotient $K^{\times}/M^{\times}$ is torsion.
    \item If $u \in M^{\times}$ is general in $K|k$, then ${k(u)}^{\times}$ is contained in $M^{\times}$.
    \item The element $x$ is contained in $M^{\times}$ and $x$ is general in $K|k$.
  \end{enumerate}

  We claim that $M$ is a additively closed in $K$.
  As $M$ is multiplicatively closed and $M^{\times} = M \smin \{0\}$ is a subgroup of $K^{\times}$, it suffices to prove that for all $u \in M$, one has $1 + u \in M$.
  As $k(x) \subset M$, we may also assume that $u \notin k(x)$, hence $u$ and $x$ are algebraically independent over $k$.

  By Fact~\ref{fact:anab/birational-bertini}, there exists $b \in k^{\times}$ and $c \in k$ such that the following elements are all general in $K|k$:
  \[ A_1 :=  \frac{u}{b \cdot x + c}, \ \ A_2 := \frac{2 \cdot u}{b \cdot x + c+1}, \ A_3 := \frac{2 \cdot u + b \cdot x + c + 1}{u + b \cdot x + c}. \]
  It is clear from the above properties that $A_1,A_2 \in M$.
  Hence
  \[ B_1 := (b \cdot x + c) \cdot (A_1 + 1) = u + b \cdot x + c \]
  and
  \[ B_2 := (b \cdot x + c+1) \cdot (A_2 + 1) = 2 \cdot u + b \cdot x + c + 1 \]
  are also elements of $M$, so that $A_3 = B_2 / B_1$ is an element of $M$ as well.
  As $A_3$ is general in $K|k$, we see that
  \[ (A_3 - 1) \cdot B_1 = 1 + u \]
  is indeed an element of $M$, as contended.

  The argument above shows that $M$ is a subfield of $K$ which contains $k$ while $K^{\times}/M^{\times}$ is torsion.
  Since $K|k$ is a function field and $k$ has characteristic $0$, it follows that $K = M$, and this concludes the proof of the proposition.
\end{proof}

\subsection{Coliniation}\label{subsection:coliniation}

Suppose now that $\phi \in \Isom^{\acl}_{\rat}(\KL(K|k),\KL(L|l))$ is synchronized, so that $\phi$ restricts to an isomorphism $K^{\times}/k^{\times} \cong L^{\times}/l^{\times}$ by Proposition~\ref{proposition:multiplicative-syncro}.
Note that $K^{\times}/k^{\times}$ is the projectivization of $K$ as a $k$-module, and similarly for $L|l$.

\begin{proposition}\label{proposition:coliniation}
  In the above context, the isomorphism
  \[ \phi : K^{\times}/k^{\times} \cong L^{\times}/l^{\times} \]
  is a coliniation, i.e.~$\phi$ sends projective lines in $K^{\times}/k^{\times}$ to projective lines in $L^{\times}/l^{\times}$.
\end{proposition}
\begin{proof}
  For $x \in K^{\times}/k^{\times}$, $x \neq 1$ and $\tilde x \in K^{\times}$ a representative of $x$, write
  \[ \Lcal_{x} := \frac{(k + \tilde x \cdot k) \cap K^{\times}}{k^{\times}} \]
  for the unique projective line in $K^{\times}/k^{\times}$ containing $1$ and $x$.
  Since $\phi$ is compatible with multiplication, it suffices to show that $\phi \Lcal_{x} = \Lcal_{\phi x}$ for all such $x$.

  Assume first that $x \in K \smin k$ and $y \in L \smin l$ are such that $\phi \Lcal_{x^{\circ}} = \Lcal_{y^{\circ}}$.
  Let $t \in K$ be algebraically independent from $x$ over $k$, and choose $u \in L \smin l$ such that $\phi t^{\circ} = u^{\circ}$.
  Choose a divisorial valuation $v$ of $K|k$ such that $v$ is trivial on $\acl_{K}(x)$ and on $\acl_{K}(t)$, and such that $x$ and $t$ have the same image in ${(Kv)}^{\times}/k^{\times}$; such a $v$ exists since $x$ and $t$ are algebraically independent over $k$.
  Put $w = v^{\phi}$, as in Fact~\ref{fact:anab/local-theory}.
  The same fact implies that $y$ and $u$ have the same image in ${(Lw)}^{\times}/l^{\times}$, while $w$ is trivial on $\acl_{L}(y)$ and on $\acl_{L}(u)$ since $\phi$ is compatible with $\acl$ (see Lemma~\ref{lemma:dim-one-geom}).

  Note that both maps
  \[ \acl_K(x)/k^\times \rightarrow {(Kv)}^\times/k^\times \leftarrow \acl_K(t)/k^\times \]
  are injective.
  Letting $\bar x = \bar t$ denote the images of $x$ and $t$ in ${(Kv)}^{\times}/k^{\times}$, the images of $\Lcal_{x^{\circ}}$ and $\Lcal_{t^{\circ}}$ both agree with $\Lcal_{\bar x}$ under these two maps.
  Since $\Uscr_{v}^{1} \cap (K^{\times}/k^{\times}) = \Urm_{v}^{1} \cdot k^{\times}/k^{\times}$ and ${\acl_{K}(t)}^{\times}/k^{\times} = \Kscr_{t} \cap (K^{\times}/k^{\times})$, we deduce that
  \[ \Lcal_{t^{\circ}} = \Kscr_{t} \cap (K^\times/k^\times) \cap (\Lcal_{x^{\circ}} \cdot (\Uscr_{v}^{1} \cap (K^\times/k^\times))). \]
  Since $\phi$ identifies $\Kscr_{t}$ with $\Kscr_{u}$, $K^{\times}/k^{\times}$ with $L^{\times}/l^{\times}$, $\Lcal_{x^{\circ}}$ with $\Lcal_{y^{\circ}}$, and $\Uscr_{v}^{1}$ with $\Uscr_{w}^{1}$, it follows that $\phi \Lcal_{t^{\circ}} = \Lcal_{u^{\circ}}$.
  Finally, since $\phi$ is synchronized, there exists some $x$ and $y$ as above such that $\phi \Lcal_{x^{\circ}} = \Lcal_{y^{\circ}}$.
  Thus, the argument above shows that for every $t \in K$ algebraically independent from $x$, we have $\phi \Lcal_{t^{\circ}} = \Lcal_{\phi t^{\circ}}$.

  If $t \in K \smin k$ is not algebraically dependent from $x$, simply choose some $s \in K \smin k$ which is algebraically independent from $x$, so the argument above shows $\phi \Lcal_{s^{\circ}} = \Lcal_{\phi s^{\circ}}$.
  Since $t$ is independent from $s$, the argument above then shows that $\phi \Lcal_{t^{\circ}} = \Lcal_{\phi t^{\circ}}$.
\end{proof}

\subsection{Concluding the proof}\label{subsection:conclusion}

We now conclude the proof of Theorem~\ref{theorem:main-anab}.
\begin{proposition}\label{proposition:ftpg}
  Assume that $\phi \in \Isom_{\rat}^{\acl}(\KL(K|k), \KL(L|l))$ is synchronized.
  Then there exists a unique isomorphism of fields $\Gamma_{\phi} : K \cong L$ such that $\phi t^{\circ} = {(\Gamma_{\phi} t)}^{\circ}$ for all $t \in K^{\times}$.
\end{proposition}
\begin{proof}
  Since $\phi$ is synchronized, it induces an isomorphism
  \[ \phi : K^{\times}/k^{\times} \cong L^{\times}/l^{\times} \]
  by Proposition~\ref{proposition:multiplicative-syncro} which is a coliniation by Proposition~\ref{proposition:coliniation}.
  By the \emph{fundamental theorem of projective geometry} (see~\cites{MR1009557,MR1783451}), there exists a unique isomorphism of fields $\gamma : k \cong l$ and a $\gamma$-semilinear isomorphism $\Gamma : K \cong L$  of additive groups such that ${\Gamma(x)}^{\circ} = \phi x^{\circ}$ for all $x \in K^{\times}$.
  Moreover, $\Gamma$ is unique with these properties up-to homothethies.
  Note that we must have $\Gamma(1) \in l^{\times}$ since $\phi(1) = 1$.
  Replace $\Gamma$ with $(1/\Gamma(1)) \cdot \Gamma$ to assume furthermore that $\Gamma(1) = 1$, and note that this $\Gamma$ is then the unique $\gamma$-semilinear isomorphism $K \cong L$ satisfying $\Gamma(1) = 1$ and ${\Gamma(x)}^{\circ} = \phi x^{\circ}$ for all $x \in K^{\times}$.

  We follow an argument which is similar to~\cite{MR2421544}*{Theorem 7.3} to show that this $\Gamma$ is a field isomorphism.
  First, since $\Gamma(1) = 1$, it follows that $\Gamma$ restricts to $\gamma : k \cong l$ on $k$.
  In particular, for $x \in K$ and $a \in k$, one has
  \[ \Gamma(a \cdot x) = \gamma(a) \cdot \Gamma(x) = \Gamma(a) \cdot \Gamma(x). \]
  Also, $\Gamma$ is already an additive homomorphism, so we only need to show its compatibility with multiplication.

  Assume therefore that $x,y \in K$ are given.
  We must show that $\Gamma(x \cdot y) = \Gamma(x) \cdot \Gamma(y)$.
  Since $\Gamma$ restricts to the field isomorphism $\gamma$ on $k$, we may assume furthermore that $x \cdot y$ and $y$ are $k$-linearly independent.
  Since $\Gamma$ induces $\phi : K^{\times}/k^{\times} \cong L^{\times}/l^{\times}$ which is multiplicative, there exists some $c \in l^{\times}$ such that
  \[ \Gamma(x \cdot y) = c \cdot \Gamma(x) \cdot \Gamma(y). \]
  Since $x \cdot y$ and $y$ are $k$-linearly independent, we see that $c^{-1} \cdot \Gamma(x \cdot y) = \Gamma(x) \cdot \Gamma(y)$ and $\Gamma(y)$ are again $l$-linearly indepdendent.

  Consider $\Gamma(x \cdot y + y)$.
  On the one hand, we have
  \[ \Gamma(x \cdot y + y) = \Gamma(x \cdot y) + \Gamma(y) = c \cdot \Gamma(x) \cdot \Gamma(y) + \Gamma(y), \]
  and on the other there exists some $d \in l^{\times}$ such that
  \begin{align*}
    \Gamma(x \cdot y + y) = \Gamma((x+1) \cdot y) &= d \cdot \Gamma(x+1) \cdot \Gamma(y) \\
    &= d \cdot (\Gamma(x)+1) \cdot \Gamma(y)\\
    &= d \cdot \Gamma(x)\cdot \Gamma(y) + d \cdot \Gamma(y)
  \end{align*}
  In particular, we see that $c = d = 1$, and hence $\Gamma(x \cdot y) = \Gamma(x) \cdot \Gamma(y)$, as required.
\end{proof}

We now conclude the proof of Theorem~\ref{theorem:main-anab}.
Let $\phi \in \Isom^{\acl}_{\rat}(\KL(K|k),\KL(L|l))$ be given.
By Proposition~\ref{proposition:rational-syncro}, there exists some $\epsilon \in \Lambda^{\times}$ such that $\psi := \epsilon \cdot \phi$ is synchronized, while Proposition~\ref{proposition:ftpg} shows there is a unique isomorphism $\Gamma_{\psi} : K \cong L$ of fields satisfying $\psi t^{\circ} = {(\Gamma_{\psi} t)}^{\circ}$.
If furthermore $\phi$ arises from a given isomorphism $\Gamma : K \cong L$, then $\phi$ is synchronized and it is easy to see that $\Gamma = \Gamma_{\phi}$.

\begin{lemma}\label{lemma:unique-syncro}
  Let $\phi : \Isom^{\acl}_{\rat}(\KL(K|k),\KL(L|l))$ be given, and suppose that $\epsilon_{1},\epsilon_{2} \in \Lambda^{\times}$ are such that both $\epsilon_{1} \cdot \phi$ and $\epsilon_{2} \cdot \phi$ are synchronized.
  Then $\epsilon_{1} = \epsilon_{2}$.
\end{lemma}
\begin{proof}
  Put $\phi_{i} := \epsilon_{i} \cdot \phi$ and $\Gamma_{i} := \Gamma_{\phi_{i}}$ as in Proposition~\ref{proposition:ftpg}.
  Consider $\Gamma = \Gamma_{1}^{-1} \circ \Gamma_{2}$ and put $\delta := \epsilon_{1}^{-1} \cdot \epsilon_{2}$.
  Write $\delta = m/n$ with $m,n$ integers such that $n > 0$.
  Note that $\Gamma$ is an automorphism of $K$ satisfying ${\Gamma(x)}^{\circ} = \delta \cdot x^{\circ}$ for all $x \in K^{\times}$.

  Let $x \in K^{\times}$ be general in $K|k$.
  Since $\Gamma k = k$, it follows that $y := \Gamma(x)$ is also general in $K|k$, and one has $y^{\circ} = \delta \cdot x^{\circ}$ hence $y^{n} = c \cdot x^{m}$ for some $c \in k^{\times}$.
  Comparing $x$-adic valuations of $x$ and $y$ with this equality, we find that $\delta \in \Zbb$ and $y = c \cdot x^{\delta}$.
  In particular, $x$ and $y$ are algebraically dependent while both $x$ and $y$ are general in $K|k$.
  This can only happen if $y \in \{c \cdot x, c \cdot x^{-1}\}$.
  The map $\epsilon \mapsto \epsilon \cdot x^{\circ}$, $\epsilon \in \Lambda^{\times}$, is injective since $x$ is general in $K|k$, hence $\delta \in \{-1,1\}$.
  We will conclude by showing $\delta = 1$.

  Assume otherwise, hence $\delta = -1$, so $y = c \cdot x^{-1}$ and there exists $d \in k^{\times}$ such that $\Gamma(1+x) = d \cdot {(1 + x)}^{-1}$.
  But $\Gamma$ is a field automorphism, hence
  \[ 1 + c \cdot x^{-1} = 1 + y = \Gamma(1+x) = d \cdot {(1+x)}^{-1}. \]
  This is clearly impossible as $x$ is transcendental over $k$.
\end{proof}

We have thus constructed a left inverse of the map
\[ \Isom(K,L) \to \Isommod^{\acl}_{\rat}(\KL(K|k),\KL(L|l)) \]
appearing in Theorem~\ref{theorem:main-anab}, which is easily seen to be functorial with respect to isomorphisms.
To conclude, we must prove that this left inverse is injective, and for this it suffices to take $K = L$ and prove that the group homomorphism we produced in this case,
\[ \Isommod^{\acl}_{\rat}(\KL(K|k),\KL(K|k)) \to \Aut(K), \]
is injective.
So suppose that $\phi$ is an automorphism of $\KL(K|k)$ which is compatible with $\acl$ and with rational submodules.
Replacing $\phi$ with $\epsilon \cdot \phi$ for some uniquely determined $\epsilon \in \Lambda^{\times}$ (see Lemma~\ref{lemma:unique-syncro}), we see may assume $\phi$ is synchronized and by assumption $\Gamma_{\phi} = \one$.
Note $\phi t^{\circ} = {(\Gamma_{\phi} t)}^{\circ} = t^{\circ}$ for all $t \in K^{\times}$, hence $\phi$ acts as the identity on $K^{\times}/k^{\times}$.
But $\KL(K|k)$ is generated as a $\Lambda$-module by this subgroup, so it follows that $\phi$ is the identity on $\KL(K|k)$ as well.
This concludes the proof of Theorem~\ref{theorem:main-anab}.

\section{Proof of the main theorem}\label{section:torelli}

We now turn to the proof of the main theorem, which we state precisely as follows.
\begin{theorem}\label{theorem:main-theorem}
  Let $\Lambda$ be a subring of $\Qbb$, and let $k$, $l$ be algebraically closed fields.
  Let $\sigma : k \hookrightarrow \Cbb$ be a complex embedding.
  Let $K$ be a function field of transcendence degree $\geq 2$ over $k$ and $L$ a function field over $l$.
  There exists an isomorphism $K \cong L$ of fields which restricts to an isomorphism $k \cong l$ if and only if there exists a complex embedding $\tau : l \hookrightarrow \Cbb$, and an isomorphism of mixed Hodge structures
  \[ \phi : \HH^1(K|k,\Lambda(1)) \cong \HH^1(L|l,\Lambda(1)) \]
  such that the map $\phi$ induces a bijection $\Rcal(K|k,\Lambda) \cong \Rcal(L|l,\Lambda)$.
  Here $\HH^{*}(K|k,\Lambda(*))$ is computed via $\sigma$ while $\HH^{*}(L|l,\Lambda(*))$ is computed via $\tau$.
\end{theorem}

The rest of the section is devoted to proving this theorem.
First, recall that any isomorphism of fields $K \cong L$ restricts to an isomorphism $k \cong l$.
Given such an isomorphism of fields, say $\Gamma : K \cong L$ with restriction $\gamma : k \cong l$, one can define $\tau : l \hookrightarrow \Cbb$ as the composition
\[ l \xrightarrow{\gamma^{-1}} k \xrightarrow{\sigma} \Cbb, \]
and the existence of $\phi$ as in the statement of the theorem is then trivial.

Let us now fix an isomorphism $\phi$ as in the statement of the theorem.
Our goal is to produce an isomorphism of fields $K \cong L$, which will automatically restrict to $k \cong l$ as observed above.
The strategy is as follows:
\begin{enumerate}
  \item First show that $\phi$ restricts to an isomorphism $\phi : \KL(K|k) \cong \KL(L|l)$ via Kummer theory.
        This will follow from the compatibility with the mixed Hodge structures.
  \item Second, show that this isomorphism is compatible with $\acl$ and with rational submodules.
  \item Conclude by applying Theorem~\ref{theorem:main-anab} to obtain an isomorphism $K \cong L$.
\end{enumerate}
From now on, we put ourselves in the context of Theorem~\ref{theorem:main-theorem}, and fix $\phi$ as in the statement.

\subsection{Compatibility with Kummer theory}

Note that the Kummer map
\[ \kappa_{K}^{\Lambda} : \KL(K|k,\Lambda) \to \HH^{1}(K|k,\Lambda(1)) \]
is injective since $\Lambda$ is flat over $\Zbb$.

\begin{lemma}\label{lemma:key-lemma}
  Let $\gamma : \Lambda \to \HH^{1}(K|k,\Lambda(1))$ be a morphism of $\Lambda$-modules.
  Then $\gamma$ arises from a morphism of mixed Hodge structures, where $\Lambda$ is the underlying module of $\Lambda(0)$, if and only if $\gamma(1)$ is contained in the image of $\kappa_{K}^{\Lambda}$.
\end{lemma}
\begin{proof}
  First suppose $t \in K^{\times}$ is given, and consider the map
  \[ \gamma_{t} : \Lambda \to \HH^{1}(K|k,\Lambda(1)) \]
  given by the composition
  \[ \Lambda = \HH^{1}(\Gbb_{m},\Lambda(1)) \xrightarrow{t^{*}} \HH^{1}(U,\Lambda(1)) \to \HH^{1}(K|k,\Lambda(1)), \]
  where $U$ is any model of $K|k$ with $t \in \Ocal^{\times}(U)$ and the $t^{*}$ is the map on cohomology induced by the associated morphism $t : U \to \Gbb_{m}$.
  Since both morphisms in this composition are compatible with the mixed Hodge structures, the same holds for $\gamma_{t}$.
  If $\gamma$ is as in the statement of the lemma, and $\gamma(1)$ is in the image of $\kappa_{K}^{\Lambda}$, then it is a linear combination of morphisms of the form $\gamma_{t}$ for $t \in K^{\times}$, so again $\gamma_{t}$ is compatible with the mixed Hodge structures.

  Conversely, suppose that $\gamma$ is compatible with the mixed Hodge structures, and let $X$ be a smooth projective model of $K|k$ and $U$ a sufficiently small nonempty open $k$-subvariety of $X$ such that $\gamma$ factors through $\HH^{1}(U,\Lambda(1)) \hookrightarrow \HH^{1}(K|k,\Lambda(1))$.
  Consider the Picard $1$-motive $\Mbf^{1,1}(U)$ associated to the inclusion $U \hookrightarrow X$, as well as the $1$-motive $\Zbb := [\Zbb \to 0]$.
  Since $\Lambda$ is flat over $\Zbb$, by Theorem~\ref{theorem:picard/picard-hodge} and Lemma~\ref{lemma:hodge-realization-iso}, the Hodge realization functor induces a canonical isomorphism
  \[ \Hom_{k}(\Zbb,\Mbf^{1,1}(U)) \otimes_{\Zbb} \Lambda \xrightarrow{\cong} \Hom_{\MHS}(\Lambda(0),\HH^{1}(U,\Lambda(1))). \]
  Unfolding the definitions, we see that
  \[ \Hom_{k}(\Zbb,\Mbf^{1,1}(U)) \otimes_{\Zbb} \Lambda = (\Ocal^{\times}(U)/k^{\times}) \otimes_{\Zbb} \Lambda, \]
  and so $\gamma$ gives rise to some element $\alpha \in (\Ocal^{\times}(U)/k^{\times}) \otimes_{\Zbb} \Lambda \subset \KL(K|k)$ via this bijection.
  Tracing through the definitions, it is easy to see that $\gamma(1) = \kappa_{K}^{\Lambda}(\alpha)$.
\end{proof}

As a consequence of this lemma, we deduce that $\phi$ restricts to an isomorphism
\[ \phi : \KL(K|k) \cong \KL(L|l) \]
which fits in the following commutative diagram:
\begin{equation}\label{equation:kummer-compat}
  \begin{tikzcd}
    \HH^{1}(K|k,\Lambda(1)) \ar[r,"\phi"] & \HH^{1}(L|k,\Lambda(1)) \\
    \KL(K|k) \ar[u,hook,"\kappa_{K}^{\Lambda}"] \ar[r,"\phi"'] & \KL(L|l) \ar[u,hook,"\kappa_{L}^{\Lambda}"']
  \end{tikzcd}
\end{equation}

\subsection{Concluding the proof}

Our final two tasks are to show that the isomorphism
\[ \phi : \KL(K|k) \cong \KL(L|l) \]
discussed in~\eqref{equation:kummer-compat} is compatible with $\acl$ and with rational submodules.
We prove these in two separate lemmas.

\begin{lemma}\label{lemma:acl-compat}
  In the above context, $\phi : \KL(K|k) \cong \KL(L|l)$ is compatible with $\acl$.
\end{lemma}
\begin{proof}
  Let $x,y \in \KL(K|k)$ be given.
  By Lemma~\ref{lemma:cup-alg-indep}, we see that $x$, $y$ are dependent if and only if $\kappa_{K}^{\Lambda}x \cup \kappa_{K}^{\Lambda} y = 0$ in $\HH^{2}(K|k,\Lambda(2))$.
  The lemma follows from the compatibility of $\phi$ with $\Rcal$ and the commutativity of~\eqref{equation:kummer-compat}.
\end{proof}

We will need the following lemma to deduce the compatibility with rational submodules.
\begin{lemma}\label{lemma:rational-helper}
  Let $t \in K \smin k$ be given.
  Then $\acl_{K}(t)|k$ is rational if and only if the Kummer map
  \[ \kappa_{\acl_{K}(t)}^{\Lambda} : \KL(\acl_{K}(t)|k) \to \HH^{1}(\acl_{K}(t)|k,\Lambda(1)) \]
  is an isomorphism.
\end{lemma}
\begin{proof}
  Put $E := \acl_{K}(t)$, and let $C$ be the smooth projective model of $E|k$.
  The Kummer map $\kappa_{E}^{\Lambda}$ in question fits in the following commutative diagram with exact rows/columns:
  \[
    \begin{tikzcd}
      {} & {} & 0 \ar[d] & 0 \ar[d] \\
      {} & {} & \KL(E|k) \ar[d,"\kappa_{E}^{\Lambda}"'] \ar[r,equal] & \KL(E|k) \ar[d,"\div \otimes \Lambda"] \\
      0 \ar[r] & \HH^{1}(C,\Lambda(1)) \ar[r] \ar[d,equal] & \HH^{1}(E|k,\Lambda(1)) \ar[r] \ar[d] & \Div^{0}(C) \otimes_{\Zbb} \Lambda \ar[r] \ar[d] & 0 \\
      0 \ar[r] & \HH^{1}(C,\Lambda(1)) \ar[r] & \coker(\kappa_{E}^{\Lambda}) \ar[r] \ar[d] & \Pic^{0}(C) \otimes_{\Zbb} \Lambda \ar[d] \ar[r] & 0 \\
      {}& {} & 0 & 0
    \end{tikzcd}
  \]
  where the map $\HH^{1}(E|k,\Lambda(1)) \to \Div^{0}(C) \otimes_{\Zbb} \Lambda$ is given by the sum of the residue maps $\partial_{x}$ associated to $x \in C(k)$.
  Now, $E|k$ is rational if and only if $C$ has genus $0$, which is equivalent to $\Pic^{0}(C) \otimes_{\Zbb} \Lambda = 0$ and $\HH^{1}(C,\Lambda(1)) = 0$.
  The assertion of the lemma follows from the exactness of the bottom row in the above diagram.
\end{proof}

\begin{lemma}\label{lemma:rational-compat}
  In the above context, $\phi : \KL(K|k) \cong \KL(L|l)$ is compatible with rational submodules.
\end{lemma}
\begin{proof}
  Let $x$ be a general element of $K|k$ so that $\Kscr_{x} := \KL(k(x)|k)$ is a rational submodule of $\KL(K|k)$.
  Put $\Lscr := \phi \Kscr$, a submodule of $\KL(L|l)$.
  Since $\phi$ is compatible with $\acl$ by Lemma~\ref{lemma:acl-compat}, $\Lscr$ must have the form $\KL(\acl_{L}(y)|l)$ for some $y \in L \smin l$ by Lemma~\ref{lemma:dim-one-geom}.
  We must show that $\acl_{L}(y)$ is rational over $l$.

  By Fact~\ref{fact:cohom-dim} and Proposition~\ref{proposition:geom-submod-cup}, we see that the image of the injective map
  \[ \HH^{1}(\acl_{L}(y)|l,\Lambda(1)) \to \HH^{1}(L|l,\Lambda(1)) \]
  can be characterized as the collection of elements $\alpha$ of $\HH^{1}(L|l,\Lambda(1))$ such that for all $\beta \in \KL(\acl_{L}(y)|l)$, one has $\kappa_{L}^{\Lambda}\beta \cup \alpha = 0$.
  The image of
  \[ \HH^{1}(k(x)|k,\Lambda(1)) \to \HH^{1}(K|k,\Lambda(1)) \]
  can be similarly characterized as the elements which pair trivially with $\Kscr_{x} = \KL(k(x)|k)$.
  The compatibility of $\phi$ with $\Rcal$ ensures that $\phi$ restricts to a commutative diagram whose horizontal morphisms are injective and whose vertical morphisms are isomorphisms which are all induced by $\phi$:
  \[
    \begin{tikzcd}
      \Kscr_{x} = \KL(k(x)|k) \ar[r,hook,"\eta"] \ar[d,"\cong"'] & \HH^{1}(k(x)|k,\Lambda(1)) \ar[r,hook] \ar[d,"\cong"] & \HH^{1}(K|k,\Lambda(1)) \ar[d,"\cong"] \\
      \Lscr = \KL(\acl_{L}(y)|k) \ar[r,hook,"\gamma"] & \HH^{1}(\acl_{L}(y)|l,\Lambda(1)) \ar[r,hook] & \HH^{1}(K|k,\Lambda(1))
    \end{tikzcd}
  \]
  The fact that $\Lscr$ is a rational submodule now follows easily from Lemma~\ref{lemma:rational-helper}.
\end{proof}

Putting everything together, $\phi$ induces an isomorphism $\KL(K|k) \cong \KL(L|l)$ via Kummer theory (see~\eqref{equation:kummer-compat}) which is compatible with $\acl$ by Lemma~\ref{lemma:acl-compat} and with rational submodules by Lemma~\ref{lemma:rational-compat}.
Thus $\Isom^{\acl}_{\rat}(\KL(K|k),\KL(L|l))$ is nonempty, hence $\Isom(K,L)$ is nonempty by Theorem~\ref{theorem:main-anab}.
This concludes the proof of Theorem~\ref{theorem:main-theorem} as any isomorphism $K \cong L$ restricts to an isomorphism $k \cong l$.

\appendix
\section{The local theory}

The \emph{local theory} in ``almost-abelian'' anabelian geometry has been extensively developed by Bogomolov~\cite{MR1260938}, Bogomolov-Tschinkel~\cite{MR1977585}, Pop~\cite{MR2735055}, and the author~\cites{TopazCrelle,MR3552293}.
Despite this, the precise statement which is needed in the above paper hasn't appeared in the literature, since previous results have mostly focused on the ``classical'' anabelian point of view of decomposition and inertia groups in Galois groups (of function fields, in this case).
In this appendix we give an essentially self-contained account of the \emph{local theory}, which is required in the main body of the present paper.
The arguments we give here are merely a distillation of the ideas developed in the references mentioned above.

We use the notation introduced in the body of the paper.
The main result in the local theory reads as follows.

\begin{theorem}\label{theorem:localtheory}
  Let $K|k$ and $L|l$ be two function fields over algebraically closed fields, and let $\Lambda$ be a subring of $\Qbb$.
  Assume that $\trdeg(K|k) \geq 2$.
  Let
  \[ \phi : \KL(K|k) \xrightarrow{\cong} \KL(L|l) \]
  be an isomorphism of $\Lambda$-modules which is compatible with $\acl$, and let $v$ be a divisorial valuation of $K|k$.
  Then there exists a unique divisorial valuation $w$ of $L|l$ such that one has $\phi (\Uscr_v) = \Uscr_{w}$ and $\phi (\Uscr_v^1) = \Uscr_{w}^1$.
\end{theorem}

We work in the context of this theorem for the rest of the appendix.
For a field extension $M|F$, We put
\[ \Gcal(M|F) := \Hom(M^{\times}/F^{\times},\Qbb) \]
considered as a $\Qbb$-module.
We consider elements of $\Gcal(M|F)$ as morphisms $M^{\times} \to \Qbb$ which are trivial on $F^{\times}$.

For a valuation $v$ of $M$, we define
\[ \Ical_{v} := \Hom(M^{\times}/(\Urm_{v} \cdot F^{\times}),\Qbb), \ \ \Dcal_{v} := \Hom(M^{\times}/(\Urm_{v}^{1} \cdot F^{\times}), \Qbb), \]
both considered as subspaces of $\Gcal(M|F)$ with $\Ical_{v} \subset \Dcal_{v}$.

Note that the isomorphism $\phi$ from Theorem~\ref{theorem:localtheory} induces an isomorphism
\[ \psi := {(\phi^{-1})}^{*} : \Gcal(K|k) \cong \Gcal(L|l), \]
and, for a divisorial valuation $v$ of $K|k$, $\Uscr_{v}$ resp.~$\Uscr_{v}^{1}$ are orthogonal to $\Ical_{v}$ resp.~$\Dcal_{v}$ with respect to the canonical $\Lambda$-bilinear pairing $\Gcal(K|k) \times \KL(K|k) \to \Qbb$.
It therefore suffices to provide a characterization of $\Ical_{v} \subset \Dcal_{v}$ for divisorial valuations $v$ of $K|k$ in terms of the following data:
\begin{enumerate}
  \item The $\Lambda$-module $\KL(K|k)$ and the $\Qbb$-module $\Gcal(K|k)$.
  \item The canonical pairing $\Gcal(K|k) \times \KL(K|k) \to \Qbb$.
  \item The binary relation on $\KL(K|k)$ given by dependence.
\end{enumerate}

\subsection{$\acl$-pairs}

Let $f,g \in \Gcal(K|k)$ be given.
We say that $(f,g)$ is an \emph{$\acl$-pair} provided that one of the following equivalent (since $\Lambda$ is a subring of $\Qbb$) conditions hold true:
\begin{enumerate}
  \item For all dependent $x,y \in \KL(K|k)$, one has $f(x) \cdot g(y) = f(y) \cdot g(x)$.
  \item For all $k$-algebraically dependent $x,y \in K^{\times}$, one has $f(x) \cdot g(y) = f(y) \cdot g(x)$.
\end{enumerate}
A subspace $\Hcal \subset \Gcal(K|k)$ is an \emph{$\acl$-subspace} if any pair of elements of $\Hcal$ are an $\acl$-pair.

\begin{lemma}\label{lemma:decomp-acl}
  Let $v$ be a valuation of $K$, and let $f \in \Ical_{v}$, $g \in \Dcal_{v}$ be given.
  Then $(f,g)$ is an $\acl$-pair.
\end{lemma}
\begin{proof}
  Suppose $x,y \in K^{\times}$ are algebraically dependent, and let $t \in K \smin k$ be such that $x,y \in \acl_{K}(t) =: M$.
  We must show that $f(x) \cdot g(y) = f(y) \cdot g(x)$.

  Let $w$ denote the restriction of $v$ to $M$.
  If $x,y \in \Urm_{v} \cdot k^{\times}$, then the claim is trivial since $f(x) = f(y) = 0$.
  Otherwise $wM/wk$ has rational rank at least one, and by Abhyankar's inequality it must be exactly one while $\trdeg(Mw|kw) = 0$ hence $\Urm_{w} \cdot k^{\times} = \Urm_{w}^{1} \cdot k^{\times}$.
  This implies that $x,y$ must have $\Qbb$-linearly dependent images in
  \[ (M^{\times}/\Urm_{w} \cdot k^{\times}) \otimes_{\Zbb} \Qbb = (M^{\times}/\Urm_{w}^{1} \cdot k^{\times}) \otimes_{\Zbb} \Qbb. \]
  Since both $f,g$ act trivially on $\Urm_{w}^{1} \cdot k^{\times}$, it follows again that $f(x) \cdot g(y) = f(y) \cdot g(x)$.
\end{proof}

\begin{theorem}\label{theorem:main-acl}
  Let $\Hcal \subset \Gcal(K|k)$ be a subspace.
  Then $\Hcal$ is an $\acl$-subspace of $\Gcal(K|k)$ if and only if there exist a valuation $v$ of $K|k$ such that $\Hcal \subset \Dcal_{v}$ and $\Ical_{v} \cap \Hcal$ has codimension at most $1$ in $\Hcal$.
\end{theorem}
\begin{proof}
  If $v$ exists as in the statement of the theorem, then the assertion about $\Hcal$ follows from Lemma~\ref{lemma:decomp-acl}.
  Conversely, suppose that $\Hcal$ is an $\acl$-subspace.
  Let $T \subset K^{\times}$ be the orthogonal of $\Hcal$ with respect to the pairing $\Gcal(K|k) \times K^{\times} \to \Qbb$.
  By~\cite{MR910395}*{Theorem 2.16}, it suffices to prove that for all $x,y \in K^{\times} \smin T$ such that $1 + x \notin T \cup x \cdot T$ and $1 + y \notin T \cup x \cdot T$, and all $f,g \in \Hcal$, one has
  \[ f(x) \cdot g(y) = f(y) \cdot g(x). \]
  Indeed, letting $H$ denote the subgroup of $K^{\times}$ generated by $T$ and all $x \in K^{\times} \smin T$ such that $1 + x \notin T \cup x \cdot T$, it follows from~\cite{MR910395}*{Theorem 2.16} that there exists a subgroup $\tilde H$ of $K^{\times}$ containing $H$ with $[\tilde H : H] \le 2$, and a valuation $v$ of $K$ such that $\Urm_{v}^{1} \subset T$ and $\Urm_{v} \subset \tilde H$.
  The condition above would then imply that $\Ical := \Hom(K^{\times}/\tilde H,\Qbb) = \Hom(K^{\times}/H,\Qbb)$ has codimension at most one in $\Hom(K^{\times}/T,\Qbb) = \Hcal$, while $\Hcal \subset \Dcal_{v}$ and $\Ical \subset \Ical_{v}$.

  Assume toward a contradiction that this does not hold, and let $x,y \in K^{\times}$, $f,g \in \Hcal$ witness this.
  Put $\Phi := (f,g) : K^{\times} \to \Qbb^{2}$.
  By our assumption on $\Hcal$, we see that whenever $u,v \in K^{\times}$ with $u \pm v \in K^{\times}$, the triple $(\Phi(u \pm v),\Phi(u),\Phi(v))$ is colinear (in the affine sense).
  Under our assumption on $x,y$ and $f,g$, the pair $\Phi(x)$, $\Phi(y)$ is linearly independent and $\Phi(1 + x) \notin \{\Phi(1),\Phi(x)\}$, $\Phi(1+y) \notin \{\Phi(1),\Phi(y)\}$.

  Embed $\Qbb^{2} = \Abb^{2}(\Qbb)$ into $\Pbb^{2}(\Qbb)$ via $(a,b) \mapsto (1:a:b)$, and compose with this inclusion and the unique projective-linear automorphism $\Sigma$ of $\Pbb^{2}(\Qbb)$ satisfying
  \[ \Sigma(1:0:0) = (1:0:0), \ \ \Sigma(\Phi(x)) = (1:1:0), \ \ \Sigma(\Phi(y)) = (1:0:1), \]
  \[ \Sigma(\Phi(1+x)) = (0:1:0), \ \ \Sigma(\Phi(1+y)) = (0:0:1). \]
  Denote the resulting map by $\Psi : K^{\times} \to \Pbb^{2}(\Qbb)$.
  The following conditions hold:
  \begin{enumerate}
    \item For all $u,v \in K^{\times}$ with $u \pm v \in K^{\times}$, the triple $(\Psi(u),\Psi(v),\Psi(u \pm v))$ is colinear.
    \item One has $\Psi(1) = (1:0:0)$, $\Psi(x) = (1:1:0)$ and $\Psi(y) = (1:0:1)$.
    \item One has $\Psi(1+x) = (0:1:0)$ and $\Psi(1+y) = (0:0:1)$.
  \end{enumerate}
  We will show that the image of $\Psi$ contains a complete projective line, which is impossible by the construction of $\Psi$, as the image of $\Psi$ misses the image under $\Sigma$ of the line at infinity.
  For $u,v \in K^{\times}$ with $\Psi(u) \ne \Psi(v)$, write $\Lfrak(u,v)$ for the unique projective line containing $\Psi(u)$ and $\Psi(v)$.

  The proof that follows is elementary but technical, and comes down to computing intersections of pairs of lines of the form $\Lfrak(u,v)$ based on a sum/difference decomposition of the corresponding elements in $K^{\times}$.
  For example, $\Psi(1+x+y) = (1:1:1)$ since $1 + x + y = (1+x)+y = (1+y)+x$, hence $\Psi(1+x+y)$ lies in the intersection $\Lfrak(1+x,y) \cap \Lfrak(1+y,x)$, which contains the unique point $(1:1:1)$.
  In the steps below, we give the sum/difference decomposition, leaving the straightforward computation of the intersection of the corresponding lines to the reader.

  \begin{step}\label{step1}
    One has $\Psi(1+x+y) = (1:1:1)$.
  \end{step}
  \begin{proof}
    This follows from the equation $1 + x + y = (1 + x) + y = (1 + y) + x$.
  \end{proof}

  \begin{step}\label{step2}
    One has $\Psi(2+x+y) = (0:1:1)$.
  \end{step}
  \begin{proof}
    This follows from the equation
    \[ 2+x+y = 1 + (1 + x + y) = (1+x) + (1+y) \]
    and Step~\ref{step1}.
  \end{proof}

  \begin{step}\label{step3}
    For all integers $n \geq 1$, $\Psi((2-n)+x+y) = (n:1:1)$ and $\Psi((1-n)+x) = (n:1:0)$.
  \end{step}
  \begin{proof}
    We proceed by induction on $n$ with the base-case $n=1$ having been done above.
    For the inductive case, first calculate $\Psi((2-(n+1))+x+y)$ using
    \begin{align*}
      (2-(n+1))+x+y &= ((1-n)+x) + y \\
                    &= ((2-n)+x+y) - 1
    \end{align*}
    combined with the inductive hypothesis, showing that
    \[ \Psi((2-(n+1))+x+y) = (n+1:1:1). \]
    Conclude by calculating $\Psi((1-(n+1))+x)$ using
    \begin{align*}
      (1-(n+1))+x &= ((2-(n+1))+x+y) - (1+y) \\
                  &= ((1-n)+x) - 1,
    \end{align*}
    along with the calculation above and the inductive hypothesis to obtain
    \[ \Psi((1-(n+1))+x) = (n+1:1:0). \]
    This concludes the proof of this step.
  \end{proof}

  \begin{step}\label{step4}
    For all integers $n,m \geq 1$, one has
    \[ \Psi((1+m-n)+m \cdot x+y) = (n:m:1), \ \Psi((m-n)+m \cdot x) = (n:m:0). \]
  \end{step}
  \begin{proof}
    Proceed by induction on $m$ with the case $m = 1$ taken care of by Step~\ref{step3}.
    For the inductive case, first use
    \begin{align*}
      (1+(m+1)-n)+(m+1) \cdot x+y &= ((1+m-n)+m \cdot x+y) + (1+x) \\
                           &= ((m-n)+m \cdot x) + (2+x+y)
    \end{align*}
    along with the inductive hypothesis and Step~\ref{step1} to deduce
    \[ \Psi((1+(m+1)-n)+(m+1) \cdot x+y) = (n:m+1:1). \]
    Then use the equation
    \begin{align*}
      ((m+1)-n)+(m+1) \cdot x &= ((m-n)+m \cdot x) + (1+x) \\
                       &= ((1+(m+1)-n)+(m+1) \cdot x+y)-(1+y).
    \end{align*}
    with the inductive hypothesis and the above calculation to deduce
    \[ \Psi(((m+1)-n)+(m+1) \cdot x) = (n:m+1:0). \]
    This concludes the proof of this step.
  \end{proof}

  \begin{step}\label{step5}
    One has $\Psi(2+x) = (1:-1:0)$.
  \end{step}
  \begin{proof}
    Use the equation $2 + x = 1 + (1 + x) = (2 + x + y) - y$ along with Step~\ref{step2}.
  \end{proof}

  We can now conclude, as follows.
  First, by Step~\ref{step4}, the image of $\Psi$ contains the set
  \[ \{(1:a:0) \ | \ a \in \Qbb_{\geq 0} \} \cup \{(0:1:0)\}. \]
  Repeat the argument above while replacing $x$ with $x' := -2-x$ and $1 + x' = -1-x$, while noting that $\Psi(x') = \Psi(2+x) = (1:-1:0)$ by Step~\ref{step5} and $\Psi(1+x') = \Psi(1+x) = (0:1:0)$, to see that $\{(1:a:0) \ | \ -a \in \Qbb_{\geq 0}\}$ is also contained in the image of $\Psi$.
  Hence the whole projective line $\Lfrak(1,x)$ is contained in the image of $\Psi$, which is impossible as discussed above.
  This provides the required contradiction, hence proving the theorem.
\end{proof}

We will need the following refinement of the above theorem.
\begin{proposition}\label{proposition:defectless-acl}
  Put $d := \trdeg(K|k)$, and let $\Hcal \subset \Gcal(K|k)$ be an $\acl$-subspace of dimension $d$.
  Then there exists a valuation $v$ of $K$ with no transcendence defect in $K|k$ such that $\Ical_{v} \subset \Hcal \subset \Dcal_{v}$ and $\Ical_{v}$ has codimension at most one in $\Hcal$.
\end{proposition}
\begin{proof}
  By Theorem~\ref{theorem:main-acl}, there exists a valuation $v$ such that $\Hcal \subset \Dcal_{v}$ and such that $\Ical_{v} \cap \Hcal$ has codimension at most one in $\Hcal$.
  By Abhyankar's inequality, $\Ical_{v}$ has dimension $\le d$.
  Furthermore, if $\dim_{\Qbb}\Ical_{v} = d$ then $v$ is defectless and $\trdeg(Kv|kv) = 0$ so that $\Dcal_{v} = \Ical_{v}$ hence $\Ical_{v} = \Hcal = \Dcal_{v}$.
  Otherwise $\dim_{\Qbb} \Ical_{v} \le d-1$, and we still have $\Ical_{v} \subset \Hcal$ since $\Ical_{v} \cap \Hcal$ has dimension at least $d-1$.
  In any case, we have $\Ical_{v} \subset \Hcal \subset \Dcal_{v}$.

  To conclude we must show that $v$ is defectless.
  If $\Ical_{v} = \Hcal$ then we are done, as observed above, so assume that $\Ical_{v} \neq \Hcal$.
  The argument above shows that $\Ical_{v}$ has dimension $d-1$, hence $vK/vk$ has rational rank $d-1$.
  Thus, the only way that $v$ can have any transcendence defect is if $\trdeg(Kv|kv) = 0$ in which case $\Dcal_{v} = \Ical_{v}$.
  The fact that $\Ical_{v} \subset \Hcal \subset \Dcal_{v}$ while $\Ical_{v} \neq \Hcal$ shows that this cannot happen.
\end{proof}

\subsection{Quasi-divisorial valuations}

A valuation $v$ of $K$ is called a \emph{quasi-divisorial valuation of $K|k$} provided that it is minimal with respect to the following conditions:
\begin{enumerate}
  \item One has $vK/vk \cong \Zbb$.
  \item One has $\trdeg(Kv|kv) = \trdeg(K|k)-1$.
\end{enumerate}
Note that a quasi-divisorial valuation is divisorial if and only if it is trivial on $k$.

\begin{proposition}\label{proposition:qpd-Ical}
  Let $\Ical \subset \Gcal(K|k)$ be a one-dimensional subspace, and put $d := \trdeg(K|k)$.
  Assume that $d \ge 2$.
  Then there exists a quasi-divisorial valuation $v$ of $K|k$ such that $\Ical = \Ical_{v}$ if and only if there exists two $\acl$-subspaces $\Hcal_{1},\Hcal_{2}$ of dimension $d$ such that $\Ical = \Hcal_{1} \cap \Hcal_{2}$.
\end{proposition}
\begin{proof}
  If $\Ical = \Ical_{v}$, then the existence of $\Hcal_{1}$ and $\Hcal_{2}$ is easy by choosing two independent $kv$-valuations $w_{1}$, $w_{2}$ of $Kv$ whose value group is $\Zbb^{d-1}$ with the lexicographic ordering, letting $v_{i} = w_{i} \circ v$, and taking $\Hcal_{i} = \Ical_{v_{i}}$.

  For the converse, we know from Proposition~\ref{proposition:defectless-acl} that there exist defectless valuations $v_{1}$ and $v_{2}$ such that $\Ical_{v_{i}} \subset \Hcal_{i} \subset \Dcal_{v_{i}}$, with $\Ical_{v_{i}}$ having codimension at most one in $\Hcal_{i}$.
  By our assumption on $d$, $v_{1}$ and $v_{2}$ are both nontrivial.
  Furthermore, the two valuations $v_{1},v_{2}$ must be \emph{dependent} since $\Ical \subset \Dcal_{v_{1}} \cap \Dcal_{v_{2}}$ with $\Ical$ being nontrivial, hence $\Urm_{v_{1}}^{1} \cdot \Urm_{v_{2}}^{1} \neq K^{\times}$.
  Let $v_{0}$ denote the maximal common coarsening of $v_{1},v_{2}$, so $\Urm_{v_{0}} = \Urm_{v_{1}}^{1} \cdot \Urm_{v_{2}}^{1}$, and thus $\Ical \subset \Ical_{v_{0}}$.
  Let $v$ be the coarsening of $v_{0}$ associated to the maximal convex subgroup of $v_{0}(\Ical^{\perp})$, where $\Ical^{\perp}$ is the orthogonal to $\Ical$ with respect to the pairing $\Gcal(K|k) \times K^{\times} \to \Qbb$.
  Note that $\Ical \subset \Ical_{v}$ as well, while $\Dcal_{v_{i}} \subset \Dcal_{v}$ for $i = 1,2$.
  We claim that $v$ does the job.

  As $v$ is a coarsening of $v_{i}$, $i = 1,2$, it is also defectless, hence $vK/vk$ is finitely-generated as an abelian group and $Kv|kv$ is finitely-generated as a field extension.
  Note that $\Ical \subset \Ical_{v} \subset \Ical_{v_{1}} \cap \Ical_{v_{2}} \subset \Hcal_{1} \cap \Hcal_{2} = \Ical$, hence $\Ical = \Ical_{v}$.
  As $vK/vk$ is a finitely-generated free abelian group, it follows that $vK/vk \cong \Zbb$ since $\dim_{\Qbb} \Ical = 1$.
  Since $v$ is defectless, this then implies that $\trdeg(Kv|kv) = \trdeg(K|k)-1$.
  The minimality of $v$ is ensured from the construction since $v(\Ical_{v}^{\perp})$ has no nontrivial convex subgroup.
\end{proof}

\begin{lemma}\label{lemma:dcal-characterization}
  Let $v$ be a quasi-divisorial valuation of $K|k$.
  Then $\Dcal_{v}$ is the subspace of $\Gcal(K|k)$ consisting of elements $f$ which form an $\acl$-pair with every element of $\Ical_{v}$.
\end{lemma}
\begin{proof}
  By Lemma~\ref{lemma:decomp-acl}, any element of $\Dcal_{v}$ forms an $\acl$-pair with any element of $\Ical_{v}$.
  Conversely, suppose that $f$ forms an $\acl$-pair with every element of $\Ical_{v}$.
  Let $x \in \mfrak_{v}$ be given.
  Note $\Ical_{v}$ is one-dimensional, and let $g$ be a generator of $\Ical_{v}$.
  If $g(x) \ne 0$, then $f(1+x) = 0$ since
  \[ f(1+x) \cdot g(x) = f(x) \cdot g(1+x) = 0 \]
  as $1+x \in \Urm_{v}$.
  Otherwise, the minimality of $v$ ensures that there exists some $y \in K^{\times}$ with $0 < v(y) < v(x)$ and $g(y) \ne 0$.
  Arguing as above, we have $f(1+y) = 0$, and since $v(y + x \cdot (1 + y)) = v(y)$, so that $g(y + x \cdot (1 + y)) = g(y) \neq 0$ as well, we have
  \[ f(1+x) + f(1+y) = f(1 + (y + x \cdot (1+y))) = 0 \]
  hence $f(1+x) = 0$.
  In other words, $f$ acts trivially on $\Urm_{v}^{1}$, so $f \in \Dcal_{v}$, as required.
\end{proof}

\subsection{Characterizing divisorial valuations}

Let us summarize what we have done.
\begin{enumerate}
  \item We gave a characterization of $\Ical_{v}$ for quasi-divisorial valuations $v$ of $K|k$ in terms of $\acl$-pairs in Proposition~\ref{proposition:qpd-Ical}.
  \item For every such $v$, we gave a characterization of $\Dcal_{v}$ in terms of $\acl$-pairs and $\Ical_{v}$ in Lemma~\ref{lemma:dcal-characterization}.
\end{enumerate}
In other words, we provided a characterization of $\Ical_{v} \subset \Dcal_{v}$ for all quasi-divisorial valuations $v$ of $K|k$ in terms of $\acl$-pairs in $\Gcal(K|k)$.
To conclude the proof of Theorem~\ref{theorem:localtheory}, we will provide a characterization of the divisorial valuations among the quasi-divisorial valuations using $\acl$-pairs in $\Gcal(K|k)$ and the dependence relation in $\KL(K|k)$.

\begin{lemma}\label{lemma:div-characterization}
  Let $v$ be a quasi-divisorial valuation of $K|k$ and assume that $\trdeg(K|k) \geq 2$.
  Then $v$ is divisorial if and only if there exists some $t \in K \smin k$ such that the composition
  \[ \Dcal_{v} \hookrightarrow \Gcal(K|k) \to \Gcal(\acl_{K}(t)|k) \]
  is surjective, where the morphism $\Gcal(K|k) \to \Gcal(\acl_{K}(t)|k)$ is the canonical one arising from the inclusion $\acl_{K}(t) \hookrightarrow K$.
\end{lemma}
\begin{proof}
  First, if $v$ is divisorial, we can choose some $t \in \Urm_{v} \smin k^{\times} \cdot \Urm_{v}^{1}$ so that $v$ is trivial on $\acl_{K}(t)$, and, letting $s$ denote the image of $t$ in $Kv$, we obtain an embedding $\iota : \acl_{K}(t) \to \acl_{Kv}(s)$.
  The maps in question fits in a commutative diagram of the form:
  \[
    \begin{tikzcd}
      \Dcal_{v} \ar[d,hook] \ar[r,two heads] & \Gcal(Kv|kv) \ar[r,two heads] & \Gcal(\acl_{Kv}(s)|k) \ar[d, two heads, "\iota^{*}"{pos=0.4}] \\
      \Gcal(K|k) \ar[rr,two heads] & {} & \Gcal(\acl_{K}(t)|k)
    \end{tikzcd}
  \]
  where the arrows decorated with two heads are surjective.
  Hence the map in question is indeed surjective.

  Conversely, assume that $v$ is nontrivial on $k$, and let $t \in K \smin k$ be given.
  Replacing $t$ with $a \cdot t^{\pm 1}$ for some $a \in k^{\times}$ if needed, we may assume that $v(t) > 0$.
  With this in mind, we have $1 + a \cdot t \in \Urm_{v}^{1}$ for all $a \in k^{\times}$ such that $v(a) > 0$.
  There are infinitely many such $a$ where $1 + a \cdot t$ have $\Qbb$-linearly independent images in $\Kscr_{\Qbb}(\acl_{K}(t)|k)$.
  Dually, the map $\Dcal_{v} \to \Gcal(\acl_{K}(t)|k)$ mentioned in the statement must have a cokernel of infinite rank.
\end{proof}

\subsection{Concluding the proof}

The proof of Theorem~\ref{theorem:localtheory} is a simple matter of putting everything together.
Let $\phi$ be as in the statement of the theorem.
Since $\phi$ is compatible with $\acl$, it follows that $\trdeg(L|l) = \trdeg(K|k) \geq 2$.
This $\phi$ induces an isomorphism $\psi := {(\phi^{-1})}^{*} : \Gcal(K|k) \cong \Gcal(L|l)$, and by Theorem~\ref{proposition:qpd-Ical} and Lemma~\ref{lemma:dcal-characterization} it follows that for every divisorial valuation $v$ of $K|k$ there exists some \emph{quasi-divisorial} valuation $w$ of $L|l$ such that $\psi \Dcal_{v} = \Dcal_{w}$ and $\psi \Ical_{v} = \Ical_{w}$.
Dualizing with respect to the pairing
\[ \Gcal(-) \times \KL(-) \to \Lambda, \]
we see that $\phi \Uscr_{v} = \Uscr_{w}$ and $\phi\Uscr_{v}^{1} = \Uscr_{w}^{1}$.

Arguing as in Lemma~\ref{lemma:dim-one-geom} from the main body of the paper, for every $t \in K \smin k$, there exists some $s \in L \smin l$ such that $\phi \Kscr_{t} = \Kscr_{s}$.
Note that the inclusion $\Kscr_{t} \hookrightarrow \KL(K|k)$ dualizes to $\Gcal(K|k) \to \Gcal(\acl_{K}(t)|k)$, and similarly $\Kscr_{s} \hookrightarrow \KL(L|l)$ dualizes to $\Gcal(L|l) \to \Gcal(\acl_{L}(s)|l)$.
Hence, by Lemma~\ref{lemma:div-characterization}, $w$ is divisorial.
To conclude, simply note that the valuation $w$ is uniquely determined by $\Uscr_{w}$, since the value group of $w$ is $\Zbb$.

\bibliography{refs}

\end{document}